\documentclass[a4paper,10pt,reqno]{article}

\usepackage[utf8]{inputenc}
\usepackage[T1]{fontenc}
\usepackage{lmodern}
\usepackage[english]{babel}
\usepackage{microtype}

\usepackage{amsmath,amssymb,amsfonts,amsthm}
\usepackage{mathtools,accents}
\usepackage{mathrsfs}
\usepackage{aliascnt}
\usepackage{braket}
\usepackage{bm}
\usepackage[normalem]{ulem}

\usepackage[a4paper,margin=3cm]{geometry}
\usepackage[citecolor=blue,colorlinks]{hyperref}

\usepackage{enumerate}
\usepackage{xcolor}
\makeatletter
\g@addto@macro\@floatboxreset\centering
\makeatother

\makeatletter
\def\newaliasedtheorem#1[#2]#3{
  \newaliascnt{#1@alt}{#2}
  \newtheorem{#1}[#1@alt]{#3}
  \expandafter\newcommand\csname #1@altname\endcsname{#3}
}
\makeatother

\numberwithin{equation}{section}

\newtheoremstyle{slanted}{\topsep}{\topsep}{\slshape}{}{\bfseries}{.}{.5em}{}

\theoremstyle{plain}
\newtheorem{theorem}{Theorem}[section]
\newaliasedtheorem{proposition}[theorem]{Proposition}
\newaliasedtheorem{lemma}[theorem]{Lemma}
\newaliasedtheorem{corollary}[theorem]{Corollary}
\newaliasedtheorem{counterexample}[theorem]{Counterexample}

\theoremstyle{definition}
\newaliasedtheorem{definition}[theorem]{Definition}
\newaliasedtheorem{question}[theorem]{Question}
\newaliasedtheorem{conjecture}[theorem]{Conjecture}

\theoremstyle{remark}
\newaliasedtheorem{remark}[theorem]{Remark}
\newaliasedtheorem{example}[theorem]{Example}

\newcommand{\setN}{\mathbb{N}}

\newcommand{\setR}{\mathbb{R}}

\newcommand{\R}{\mathbb{R}}

\newcommand{\eps}{\varepsilon}

\let\altphi\phi
\let\phi\varphi
\let\varphi\altphi
\let\altphi\undefined

\newcommand{\abs}[1]{\left\lvert#1\right\rvert}
\newcommand{\norm}[1]{\left\lVert#1\right\rVert}

\newcommand{\weakto}{\rightharpoonup}

\newcommand{\di}{\mathop{}\!\mathrm{d}}
\newcommand{\didi}[1]{\frac{\mathrm{d}}{\di#1}}

\newcommand{\loc}{{\rm loc}}

\newcommand{\res}{\mathop{\hbox{\vrule height 7pt width .5pt depth 0pt
\vrule height .5pt width 6pt depth 0pt}}\nolimits}

\DeclareMathOperator{\supp}{supp}

\newcommand{\Ch}{{\sf Ch}}
\newcommand{\ee}{{\rm e}}
\newcommand{\Ric}{{\rm Ric}}
\newcommand{\Opt}{\mathrm{OptGeo}}

\DeclareMathOperator{\LIP}{LIP}
\DeclareMathOperator{\Geo}{Geo}
\DeclareMathOperator{\Lip}{Lip}

\DeclareMathOperator{\LIPc}{LIP_c}

\DeclareMathOperator{\Per}{Per}
\DeclareMathOperator{\LIPloc}{LIP_{loc}}
\DeclareMathOperator{\BV}{BV}

\newcommand{\leb}{\mathscr{L}}

\newcommand{\Prob}{\mathscr{P}}
\newcommand{\Borel}{\mathscr{B}}

\newcommand{\dist}{\mathsf{d}}
\DeclareMathOperator{\Dist}{dist}
\DeclareMathOperator{\sfd}{\dist}
\newcommand{\meas}{\mathfrak{m}}
\DeclareMathOperator{\mm}{\meas}

\newcommand{\ess}{\rm ess}

\newcommand{\AC}{{\rm AC}}
\DeclareMathOperator{\CD}{CD}
\DeclareMathOperator{\RCD}{RCD}
\newfont{\tmpf}{cmsy10 scaled 2500}

\begin{document}

\title{Polya-Szego Inequality and Dirichlet $p$-Spectral gap for non-smooth spaces with Ricci curvature bounded below}
\author{Andrea Mondino\thanks{Oxford University, Mathematical Institute, Oxford, United Kingdom. email: Andrea.Mondino@maths.ox.ac.uk (corresponding author)}  \, and \, Daniele Semola \thanks{Scuola Normale Superiore, Pisa, Italy. email: daniele.semola@sns.it}}
\maketitle

\begin{abstract}

We study decreasing rearrangements of functions defined on (possibly non-smooth) metric measure spaces with Ricci curvature bounded below by $K>0$ and dimension bounded above by $N\in (1,\infty)$ in a synthetic sense, the so called $\CD(K,N)$ spaces. We first establish a Polya-Szego type inequality stating that the $W^{1,p}$-Sobolev norm decreases  under such a rearrangement and apply the result to show sharp spectral gap for the $p$-Laplace operator with Dirichlet boundary conditions (on open subsets), for every $p\in (1,\infty)$. This extends to the non-smooth setting a classical result of B\'erard-Meyer \cite{BerardMeyer} and Matei \cite{Matei00}; remarkable examples of spaces fitting our framework and for which the results seem new include: measured-Gromov Hausdorff limits of Riemannian manifolds with Ricci $\geq K>0$, finite dimensional Alexandrov spaces with curvature$\geq K>0$, Finsler manifolds with Ricci $\geq K>0$.
\\In the second part of the paper we prove new rigidity and almost rigidity results attached to the aforementioned inequalities, in the framework of $\RCD(K,N)$ spaces, which are interesting even for smooth Riemannian manifolds with Ricci $\geq K>0$.

\bigskip

\centerline{\bf R\'esum\'e}

Nous \'etudions les r\'earrangements d\'ecroissants des fonctions d\'efinies sur les espaces m\'etrique mesur\'es (\'eventuellement non lisses) avec une courbure de Ricci $\geq  K> 0 $ et une dimension $\leq  N \in (1, \infty)$ dans un sens synth\'etique, appel\'ee espaces $ \CD(K, N) $. Nous \'etablissons d'abord une in\'egalit\'e de type Polya-Szego affirmant que la norme   de Sobolev $W^{1, p}$  diminue sous un tel r\'earrangement et appliquons le r\'esultat pour montrer un \'ecart spectral  pour l'op\'erateur $p$-laplacien avec des conditions au bord de Dirichlet (sur les sous-ensembles ouverts), pour chaque $ p \in (1, \infty)$. Cela \'etend au cadre non-lisse un r\'esultat classique de B\'erard-Meyer \cite {BerardMeyer} et de Matei \cite {Matei00}; des exemples remarquables d'espaces 	satisfaisant ce cadre et pour lesquels les r\'esultats semblent nouveaux sont notamment les suivants: mGH-limites des vari\'et\'es Riemanniennes avec Ricci $\geq K> 0 $, espaces d'Alexandrov \'a dimension finie avec \mbox{courbure $\geq K> 0 $,} vari\'et\'es de Finsler avec Ricci $\geq K> 0 $.
\\ Dans la deuxi\`eme partie du article, nous montrons de nouveaux r\'esultats de rigidit\'e et presque de rigidit\'e li\'es aux in\'egalit\'es susmentionn\'ees, dans le cadre des espaces $ \RCD (K, N) $, qui sont int\'eressants m\^eme pour les vari\'et\'es Riemanniennes lisses avec Ricci $\geq K> 0 $.

\bigskip


\end{abstract}

\tableofcontents

\section{Introduction}
In 1884 Lord Rayleigh, in his book about the theory of sound \cite{Ray}, conjectured that, among all membranes of a given area, the disk has the lowest fundamental frequency of vibration.
This was proven  in 1920ies by Faber \cite{Fa23} and Krahn \cite{Kr25} for domains in the Euclidean plane and extended  by Krahn \cite{Kr26}  to higher dimensions. The celebrated Rayleigh-Faber-Krahn inequality reads as follows.

\begin{theorem}[Rayleigh-Faber-Krahn inequality \cite{Fa23,Kr25,Kr26}]
Let $\Omega\subset \R^{n}$ be a  relatively compact open domain with smooth boundary.
Then the  first Dirichlet eigenvalue of  $\Omega$ is bounded below by the first  Dirichlet eigenvalue of a Euclidean ball having the same volume of $\Omega$, moreover the inequality is rigid in the sense that equality is attained if and only if $\Omega$ is a ball.  
 \end{theorem}
The proof of the Rayleigh-Faber-Krahn inequality is based on two key facts: a variational characterisation for the first Dirichlet eigenvalue and the properties of symmetric decreasing rearrangements of functions. The variational characterisation of the first eigenvalue is given by 
\begin{equation}\label{eq:firstEV}
\lambda(\Omega):=\inf_{u\in C^{1}_{c}(\Omega)} \frac{\int_{\Omega} |\nabla u|^{2} dx}{\int_{\Omega} u^{2} dx}.
\end{equation}
Let us now briefly recall few basics about decreasing rearrangements. Given an open subset $\Omega\subset \R^{n}$, the symmetrized domain $\Omega^{*}\subset \R^{n}$ is a ball  with the same measure as $\Omega$ centred at the origin. If $u$ is a real-valued Borel function defined on  $\Omega$, its spherical decreasing rearrangement $u^{*}$ is a function defined on the ball $\Omega^{*}$ with the following properties: $u^{*}$ depends only on the distance from the origin, is decreasing along the radial direction  and is equi-measurable with $u$ (i.e. the super-level sets have the same volume: $|\{u>t\}|=|\{u^{*}>t\}|$, for every $t\in \R$).  Since the function  and its spherical decreasing rearrangement are equi-measurable, their $L^{2}$-norms are the same. The key property that Faber and Krahn  proved is that the $L^{2}$-norm of the gradient of a function decreases under rearrangements.  This last property  was formalised, extended to every $L^{p}$, $1<p<\infty$, and applied to several problems in mathematical physics by Polya and Szego in their book \cite{PS51}; probably this is why it is now well known as the Polya-Szego inequality.  The Polya-Szego inequality, combined with the variational characterization \eqref{eq:firstEV}, immediately gives the Rayleigh-Faber-Krahn inequality.
\\

Such a stream of ideas was  extended  in 1992 by B\'erard-Meyer \cite{BerardMeyer} to Riemannian manifolds $(M^{n},g)$ with $\Ric_{g}\geq K g$,  $K>0$. 
They proved the following result:

\begin{theorem}[B\'erard-Meyer \cite{BerardMeyer}]\label{thm:BeMe}
Let $(M^{n},g)$ be a Riemannian manifold  with $\Ric_{g}\geq K g$,  $K>0$, and let $\Omega\subset M$ be an open subset with smooth boundary. Let ${\mathbb S}^{n}_{K}$ be the round $n$-dimensional sphere  of radius $\sqrt{(n-1)/K}$ and let $\Omega^{*}\subset {\mathbb S}^{n}_{K}$ be a metric ball having the same renormalized volume of $\Omega$, i.e $\frac{|\Omega|}{|M|}=\frac{|\Omega^{*}|}{|{\mathbb S}^{n}_{K}|}$. Then $\lambda(\Omega)\geq \lambda(\Omega^{*})$ and equality is achieved if and only if $M$ is isometric to  ${\mathbb S}^{n}_{K}$ and $\Omega$ is a metric ball in ${\mathbb S}^{n}_{K}$.
 \end{theorem}

 The two key ideas in \cite{BerardMeyer} are the following. First, in the same spirit as above, for a function $u\in C^{1}_{c}(M)$ define a spherical decreasing rearrangement $u^{*}$ on ${\mathbb S}$; second, replace the Euclidean isoperimetric inequality by the L\'evy-Gromov isoperimetric inequality \cite[Appendix C] {Gro} in the proof of the corresponding Polya-Szego type inequality.  Let us finally mention that, arguing along the same lines, the comparison \autoref{thm:BeMe} was generalized to the first Dirichlet eigenvalue of the $p$-Laplacian for any $p\in (1,\infty)$ by Matei \cite{Matei00}.
\\

The goal of the present paper is two-fold: first, we generalise the Polya-Szego and the B\'erard-Meyer inequalities to non-smooth spaces with Ricci curvature  bounded below in a synthetic sense; second, we obtain a rigidity result for Polya-Szego inequality and an almost rigidity result for the Dirichlet $p$-spectral gap which  sound interesting even for smooth Riemannian manifolds.

\subsection{Polya-Szego and $p$-spectral gap in $\CD(K,N)$ spaces}

In order to discuss the main results of the paper let us introduce some preliminaries about non-smooth spaces with Ricci curvature  bounded below in a synthetic sense. 
\\A metric measure space (m.m.s. for short) is a triple $(X,\dist,\meas)$ where $(X,\dist)$ is a compact metric space endowed with a Borel probability measure $\mm$ with $\supp(\mm)=X$,  playing the role of reference volume measure. 
Using optimal-transport techniques,  Lott-Villani \cite{LottVillani} and Sturm \cite{Sturm06I,Sturm06II} introduced the so called curvature-dimension condition $\CD(K,N)$: the rough geometric picture is that a m.m.s. satisfying  $\CD(K,N)$ should be thought of as a possibly non-smooth metric measure space with Ricci curvature bounded below by $K\in \R$ and dimension bounded above by $N\in (1,\infty)$ in a synthetic sense.  The basic idea of this synthetic point of view is to consider weighted convexity properties of suitable entropy functionals along geodesics in the space of probability measures endowed with the quadratic transportation distance.
\\A first technical assumption throughout the paper  is the so called \emph{essentially non-branching} property \cite{RS2014}, 
which roughly amounts to require that the $L^{2}$-optimal transport between two absolutely continuous (with respect to the reference measure $\mm$) probability measures moves along a family of geodesics with no intersections, i.e.
a non-branching set of geodesics (for the precise definitions see Section \ref{SS:CDDef}).
\\The class of  essentially non-branching $\CD(K,N)$ spaces  is very natural for extending the Polya-Szego/B\'erard-Meyer  results. Indeed a key ingredient for both  is the isoperimetric inequality (via a coarea formula argument) and  it was proved by Cavalletti with the first author \cite{CavallettiMondino17} that the L\'evy-Gromov isoperimetric inequality extends to  essentially non-branching $\CD(K,N)$ spaces (see Section \ref{SS:FPSLG} for the details). 
\\Examples of essentially non-branching $\CD(K,N)$ spaces are Riemannian manifolds with Ricci curvature bounded below, finite dimensional Alexandrov spaces with curvature bounded below, Ricci limits and more generally $\RCD(K,N)$-spaces, Finsler manifolds endowed with a strongly convex norm and with Ricci bounded below; let us stress that our results are new in all these celebrated classes of spaces (apart from smooth manifolds). A standard example of a space failing to satisfy the essential non-branching property is $\R^{2}$ endowed with the $L^{\infty}$ norm.
\\

In order to state the main theorems, let us introduce some notation about the model one-dimensional space and the corresponding monotone rearrangement.
\\For any $K>0$ and $1<N<+\infty$ we define the one dimensional model space $(I_{K,N},\dist_{eu}, \meas_{K,N})$ for the curvature dimension condition of parameters $K$ and $N$ by
\begin{equation}\label{eq:defonedimIntro}
I_{K,N}:=\left(0,\sqrt{\frac{N-1}{K}}\pi\right),\quad \meas_{K,N}:=\frac{1}{c_{K,N}}\sin\left({t\sqrt{\frac{K}{N-1}}}\right)^{N-1}\leb^1\res I_{K,N},
\end{equation}
where $\dist_{eu}$ is the restriction to $I_{K,N}$ of the canonical Euclidean distance over the real line, $\leb^1$ is the standard Lebesgue measure,  and $c_{K,N}:=\int_{I_{K,N}} \sin\big(t\sqrt{K/(N-1)}\big)^{N-1} \di \leb^1(t)$ is the normalizing constant.

We now introduce the corresponding monotone rearrangement.  To this aim, given an open domain $\Omega\subset X$ and a non-negative Borel function $u:\Omega\to[0,+\infty)$ we define its distribution function $\mu:[0,+\infty)\to[0,\meas(\Omega)]$ by 
\begin{equation}\label{eq:defdistrIntro}
\mu(t):=\meas(\{u> t\}).
\end{equation}
It is not difficult to check that the distribution function $\mu$ is non increasing and left-continuous.
\\We will let $u^{\#}$ be the generalized inverse of $\mu$, defined in the following way:
\begin{equation*}
u^{\#}(s):=
\begin{cases*}
{\ess}\sup u & \text{if $s=0$},\\
\inf\left\lbrace t:\mu(t)<s \right\rbrace &\text{if $s>0$}. 
\end{cases*}
\end{equation*}

\begin{definition}[Rearrangement on one dimensional model spaces] 
Let $(X,\dist,\meas)$ be a $\CD(K,N)$ space, for some  $K>0$, $1<N<+\infty$,  and let  $\Omega\subset X$ be an open subset. Let  $\left(I_{K,N},\dist_{eu},\meas_{K,N}\right)$ be the one-dimensional model space defined in \eqref{eq:defonedimIntro} and consider  $[0,r]\subset I_{K,N}$ such that $\mm_{K,N}([0,r])=\mm(\Omega)$. For any  Borel function $u:\Omega\to [0,+\infty)$, the \emph{monotone rearrangement} $u^*:[0,r]\to \R^{+}$  is defined by
\begin{equation}\label{defu*}
u^*(x):=u^{\#}(\meas_{K,N}([0,x])), \quad \forall x\in [0,r].
\end{equation}
\end{definition}
\noindent
For an arbitrary Borel function $u:\Omega\to(-\infty+\infty)$, let $u^*$ be the monotone rearrangement of $\abs{u}$.

Finally, we denote by $W^{1,p}_{0}(\Omega)$ the closure (with respect to the $W^{1,p}$-topology) of the set of Lipschitz functions compactly supported in $\Omega$ (see Section \ref{sec:preliminaries} for more details).
\\We can now state the first main result of the paper.

\begin{theorem}[Polya-Szego inequality]\label{cor:symmapssobolevintosobolev}
Let $(X,\dist,\meas)$ be an essentially non branching $\CD(K,N)$ space for some $K>0$, $N\in (1,+\infty)$.  Let $\Omega\subset X$ be an open subset and consider  $[0,r]\subset I_{K,N}$ such that $\mm_{K,N}([0,r])=\mm(\Omega)$. 
\\Then the monotone rearrangement maps $W^{1,p}_0(\Omega)$ into $W^{1,p}\left(([0,r],\dist_{eu},\meas_{K,N})\right)$ for any $1<p<+\infty$. Moreover for any $u\in W^{1,p}_0(\Omega)$ it holds $u^{*}(r)=0$ and
\begin{equation}\label{PSIntro}
\int_{\Omega}\abs{\nabla u}_w^p\di \meas\ge\int_{0}^{r}\abs{\nabla u^*}^p\di \meas_{K,N}.
\end{equation}
\end{theorem}
\autoref{cor:symmapssobolevintosobolev} will be proved in Section \ref{sec:polyaszego}. The  two main ingredients in the proof are the coarea formula and the L\'evy-Gromov isoperimetric inequality, though the full argument requires some work and several 
intermediate results.  
\\

The second main result is a spectral gap for the $p$-Laplacian with Dirichlet boundary conditions, in the spirit of Berard-Meyer-Matei \autoref{thm:BeMe}. In order to state it we need to introduce some more notation.
\\For every $v\in (0,1)$, let $r(v)\in I_{K,N}$ be such that $v=\meas_{K,N}([0,r(v)])$. For any fixed $1<p<+\infty$, for any $v\in(0,1)$ and for any choice of $K>0$ and $1<N<+\infty$,  define
\begin{equation*}
\lambda^p_{K,N,v}:=\inf\left\lbrace\frac{\int_0^{r(v)} \abs{u'}^p\di \meas_{K,N}}{\int_0^{r(v)}u^p\di \meas_{K,N}}:\quad u\in\LIP([0,r(v)];[0,+\infty)),\quad u(r(v))=0 \text{ and $u\not\equiv 0 $} \right\rbrace.
\end{equation*}
For any metric measure space $(X,\dist,\meas)$ with $\meas(X)=1$, for any open subset $\Omega\subset X$ and for any $1<p<+\infty$,  define
\begin{equation*}
\lambda_{X}^p(\Omega):=\inf\left\lbrace\frac{\int_{\Omega} \abs{\nabla u}^p\di \meas}{\int_{\Omega}u^p\di \meas}:\quad u\in\LIPc(\Omega;[0,+\infty))\text{ and $u\not\equiv 0$}\right\rbrace.
\end{equation*}
Observe that for any $2\leq N\in\setN$ and $K>0$,  $\lambda^p_{K,N,v}=\lambda^{p}_{\mathbb{S}^{N}_{K}}(B_{v})$, where  $ \mathbb{S}^{N}_{K}$ is the round $N$-dimensional sphere of radius $\sqrt{\frac{N-1}{K}}$ and  $B_{v}\subset \mathbb{S}^{N}_{K}$ is a metric ball  of volume $v$.
\\We can now state our second main result.

\begin{theorem}[$p$-Spectral gap with Dirichlet boundary conditions]\label{thm:pSpectr}
Let $(X,\dist,\meas)$ be an essentially non branching $\CD(K,N)$  space for some $K>0$, $1<N<+\infty$, and let $\Omega\subset X$ be an open domain with $\meas(\Omega)=v\in(0,1)$. Then it holds
\begin{equation*}
\lambda^p_{X}(\Omega)\ge\lambda^p_{K,N,v}
\end{equation*}
for any $1<p<+\infty$.
\end{theorem}
The spectral gap in $\CD(K,N)$ spaces for \emph{Neumann boundary conditions}, called Lichnerowicz inequality, was established by Lott-Villani \cite{LV07} in case $p=2$ (see also \cite{EKS} and \cite{JZ16} for related results in $\RCD(K,N)$ spaces) and by Cavalletti with the first author \cite{CMGT} for general $p\in (1,\infty)$ with different techniques.

\subsection{Rigidity and almost rigidity in $\RCD(K,N)$ spaces}
In order to discuss the rigidity statements associated to \autoref{cor:symmapssobolevintosobolev} and \autoref{thm:pSpectr} let us recall the ``Riemannian'' refinement of the $\CD$ condition, called $\RCD$.
Introduced by  Ambrosio-Gigli-Savar\'e \cite{AmbrosioGigliSavare14} in case $N=\infty$ (see also \cite{AGMR12}), the $\RCD$ condition is a strengthening of the $\CD$ condition by the requirement that the Sobolev space $W^{1,2}((X,\sfd,\mm))$ is Hilbert (or, equivalently, the heat flow, or equivalently the laplacian, is linear).  The main motivation is that the  $\CD$ condition allows Finsler structures while the $\RCD$ condition isolates the ``Riemannian'' spaces.  A key property of the $\RCD$ condition is that, as well as $\CD$, is stable under measured Gromov-Haudorff convergence \cite{AmbrosioGigliSavare14,GigliMondinoSavare13}.  The finite dimensional refinement was subsequently proposed and throughly studied in \cite{Gigli1,EKS,AmbrosioMondinoSavare}  (see also \cite{CMi16}).
We refer to these papers and references therein for a general account on the synthetic formulation of the latter Riemannian-type Ricci curvature lower bounds; for a survey of results, see the Bourbaki seminar \cite{VilB} and the recent ICM-Proceeding \cite{AmbrosioICM}.

We can now state the rigidity result associated to the Polya-Szego inequality \autoref{cor:symmapssobolevintosobolev}.  In order to simplify the notation we will consider $K=N-1$, the case of a general $K>0$ follows by a scaling argument (recall that  $(X,\dist,\meas)$ is an $\RCD(K,N)$  space for some $K>0$ and $1<N<+\infty$ if and only if the rescaled space $(X,\dist',\meas)$, where $\dist':=\sqrt{\frac{N-1}{K}}\dist$, is an $\RCD(N-1,N)$ space).

\begin{theorem}[Rigidity in the Polya-Szego inequality]\label{thm:SpRigPoSzIntro}
 Let $(X,\dist,\meas)$ be an $\RCD(N-1,N)$ space for some $N\in[2,+\infty)$ with $\meas(X)=1$ and let  $\Omega\subset X$ be an open set such that $\meas(\Omega)=v\in (0,1)$.
Assume that  for some $p\in (1,\infty)$ there exists $u\in W^{1,p}_0(\Omega)$, $u\not\equiv 0$, achieving equality  in the Polya-Szego inequality \eqref{PSIntro}.
\\ Then $(X,\dist,\meas)$ is a spherical suspension, namely there exists an $\RCD(N-2,N-1)$ space $(Y,\dist_Y,\meas_Y)$ with $\meas_Y(Y)=1$ such that $(X,\dist,\meas)$ is isomorphic as a metric measure space to $[0,\pi]\times_{\sin}^{N-1}Y$.
 \\If moreover the function $u$ achieving equality in the Polya-Szego inequality \eqref{eq:polyaszego} is Lipschitz and $\abs{\nabla u}(x)\neq 0$ for $\meas$-a.e. $x\in\supp(u)$, then  $u$ is radial; i.e.  $u=f(\dist(\cdot,x_{0}))$, where $x_{0}$ is a tip of a spherical suspension structure of $X$ and $f:[0,\pi]\to \R$ satisfies $|f|=u^{*}$.
\end{theorem}

When specialized to the smooth setting, the last result reads as follows.

\begin{corollary}[Rigidity in the Polya-Szego inequality-Smooth Setting]\label{thm:SpRigPoSzIntroSmooth}
Let $(M,g)$ be an $N$-dimensional Riemannian manifold, $N\geq 2$, with $\Ric_{g}\geq (N-1) g$ and denote by $\mm$ the normalized Riemannian volume measure. Let $\Omega\subset X$ be an open subset with $\mm(\Omega) \in (0,1)$.
 \\Assume that  for some $p\in (1,\infty)$ there exists $u\in W^{1,p}_0(\Omega)$, $u\not\equiv 0$, achieving equality  in the Polya-Szego inequality \eqref{PSIntro}.
\\ Then $(M,g)$ is isometric to the round sphere ${\mathbb S}^{N}$ of constant sectional curvature one.
 \\If moreover the function $u$  achieving equality  in the Polya-Szego inequality \eqref{eq:polyaszego} is Lipschitz and $\abs{\nabla u}(x)\neq 0$ for a.e. $x\in\supp(u)$, then  $u$ is radial; i.e.  $u=f(\dist(\cdot,x_{0}))$, for some $x_{0}\in {\mathbb S}^{N}$  and $f:[0,\pi]\to \R$ satisfying $|f|=u^{*}$.
\end{corollary}

Let us mention that our proof of both \autoref{thm:SpRigPoSzIntro} and  \autoref{thm:SpRigPoSzIntroSmooth} builds on top of the almost rigidity in L\'evy-Gromov inequality  \cite{CavallettiMondino17} and seems new even in the smooth setting. The rough idea is that if the space $X$ is not a spherical suspension then by the  almost rigidity in L\'evy-Gromov inequality, there is a gap in the isoperimetric profile of $X$ and the model isoperimetric profile $I_{N-1,N}$. Thus it is not possible to achieve almost equality in the  Polya-Szego inequality  for suitable approximations $u_{n}\in \LIP_{c}(\Omega)$ of $u$ with $\abs{\nabla u_{n}}(x)\neq 0$ $\meas$-a.e. $x\in\supp(u_{n})$, hence contradicting that $u\in W^{1,p}_{0}(\Omega)$ achieves equality in Polya-Szego inequality. The rigidity statement in the function is more subtle and basically consists in proving that the  structure of spherical suspension induced by the  optimality in L\'evy-Gromov by every super-level set $\{u>t\}$ is independent of $t$.

\begin{remark}\label{rem:nablauneq0Intro}
A natural question about  \autoref{thm:SpRigPoSzIntro} regards  sharpness of the assumptions.  Clearly, if $u\equiv 0$  also the decreasing rearrangement $u^{*}$  vanishes; hence $u,u^{*}$  achieve equality in the Polya-Szego inequality but one cannot expect to infer anything on the space. 

Let us  also stress that the condition $|\nabla u|\neq 0$ $\mm$-a.e. is necessary to infer that $u(\cdot)=u^{*}\circ \dist(x_{0}, \cdot)$, even knowing a priori that the space is a spherical suspension with pole $x_{0}$ and that $u$ achieves equality in Polya-Szego inequality. Indeed let  $X={\mathbb S^{N}}$ be the round sphere, fix points $x_{1}\neq x_{2}\in \mathbb{S}^N$ and radii $0<r_1<r_2<r_3$ such that $B_{r_1}(x_1)\subset B_{r_2}(x_2)\subset B_{r_3}(x_2)$. Consider a function $u:{\mathbb S^{N}}\to [0,1]$ which is radially decreasing on $B_{r_1}(x_1)$ (with respect to the pole $x_1$), constant on $B_{r_2}(x_2)\setminus B_{r_1}(x_1)$ and radially decreasing on $B_{r_3}(x_2)\setminus B_{r_2}(x_2)$ (w.r.t. the pole $x_2$). It is easy to check that such a function $u$ achieves equality in the Polya-Szego inequality but is not globally radial.
\end{remark}

Our second rigidity result concerns the  Dirichlet $p$-spectral gap.

\begin{theorem}[Rigidity for the Dirichlet $p$-spectral gap]\label{thm:RigpSprectrIntro}
Let $(X,\dist,\meas)$ be an $\RCD(N-1,N)$ space. Let $\Omega\subset X$ be an open subset with $\meas(\Omega)=v$ for some $v\in(0,1)$ and suppose that $\lambda^p_{X}(\Omega)=\lambda^p_{N-1,N,v}$.
Then
\begin{enumerate}
\item $(X,\dist,\meas)$ is isomorphic to a spherical suspension: i.e. there exists an $\RCD(N-2,N-1)$ space $(Y,\dist_Y,\meas_Y)$ such that  $X\simeq[0,\pi]\times_{\sin}^{N-1} Y$;
\item the topological closure $\bar{\Omega}$ of $\Omega\subset X$ coincides with the closed metric ball centred at one of the tips of the spherical suspension: i.e. either  $\bar{\Omega}=[0,R]\times Y  $ or $\bar{\Omega}=[\pi-R,\pi]\times Y$, where $R\in (0,\pi)$ is such that $\meas_{N-1,N}([0,R])=v$;
\item  the eigenfunction $u\in W^{1,p}_{0}(\Omega)$ associated to $\lambda^{p}_{X}(\Omega)$ is unique up to a scalar factor and it coincides with the radial one: i.e. called   $x_{0}$ the centre of $\Omega$ and $w:[0,R]\to [0,+\infty)$ the first eigenfunction on $([0,R], \dist_{{eu}}, \mm_{N-1,N})$ corresponding to $\lambda^{p}_{N-1,N,v}$ (i.e. with the constraint $w(R)=0$), it holds   that $u(\cdot)=w\circ \dist(x_{0}, \cdot)$.
\end{enumerate}
 \end{theorem}

The proof of \autoref{thm:RigpSprectrIntro} builds on top of the rigidity in the L\'evy-Gromov inequality proved in \cite{CavallettiMondino17}; indeed the rough idea to establish the first and second assertions is to prove that if $\lambda^p_{X}(\Omega)=\lambda^p_{N-1,N,v}$, then the super-level sets of the first $p$-eigenfunction are optimal in the L\'evy-Gromov inequality. The proof of  the third assertion requires more work.  The rough idea is to show that the first Dirichlet $p$-eigenfunction is unique thus, knowing already that $\Omega$ is almost a ball centred at a tip of the spherical suspension and hence there is already a natural \emph{radial}  first Dirichlet $p$-eigenfunction suggested by the model space, it follows that $u$ must be radial. In the proof of the uniqueness of the first Dirichlet $p$-eigenfunction we have been inspired by a paper of Kawhol-Lindqvist  \cite{KawohlLindqvist06} dealing with smooth Riemannian manifolds and, in order to implement the arguments in non-smooth setting, we make use of the theory of tangent modules of  m.m.s. developed by Gigli \cite{Gigli14} (after Weaver \cite{Weaver}).
\\

Let us also mention that the rigidity for the \emph{Neumann} spectral gap,  known as Obata Theorem,  was established in case $p=2$ by Ketterer \cite{K15} and by Cavalletti with the first author \cite{CMGT} for general $p\in (1,\infty)$.
\\

We conclude the introduction with an almost-rigidity result which seems interesting even in the smooth framework, i.e. if  $(X,\dist,\meas)$ is an $N$-dimensional Riemannian manifold with Ricci curvature bounded below by $N-1$. Let us point out that in \cite{Bertrand05} some related almost rigidity results have been obtained in the smooth setting for $p=2$ under the  additional assumption that the domain $\Omega$ is (mean) convex and $\mm(\Omega)\leq 1/2$.\\
A key point in the proof is that the class of $\RCD(N-1,N)$ spaces is compact with respect to mGH convergence, a fact which clearly fails in the smooth setting as the limits usually present singularities.
\\We denote by $\sfd_{mGH}$ the measured Gromov Hausdorff distance between two normalized compact  metric measure spaces.

\begin{theorem}[Almost rigidity in the $p$-spectral gap]\label{thm:almostrigidity}
Fix $2\leq N<+\infty$ and $v\in(0,1)$. Then, for any $\epsilon>0$, there exists $\delta=\delta(v,N)>0$ with the following property: let $(X,\dist,\meas)$ be an $\RCD(N-1,N)$ m.m.s. with $\meas(X)=1$ and $\Omega\subset X$ be an open domain with $\meas(\Omega)=v$ and $\lambda^{p}_{X}(\Omega)<\lambda^p_{N-1,N,v}+\delta$.

 Then there exists a spherical suspension $(Y,\dist_Y,\meas_Y)$ (i.e. there exists an $\RCD(N-2,N-1)$ space $(Z,\dist_Z,\meas_Z)$ with $\meas_Z(Y)=1$ such that $Y$ is isomorphic as a metric measure space to $[0,\pi]\times_{\sin}^{N-1}Z$) such that
\begin{equation*}
\dist_{mGH}\left((X,\dist,\meas),(Y,\dist_Y,\meas_Y)\right)<\epsilon.
\end{equation*}
\end{theorem}

\noindent{\bf Acknowledgement:}  
The work was developed while A.M. was based in the Mathematics
Institute at the University of Warwick and, partly, when  D.S. was visiting the institute. They would like to thank the institute for the excellent working
conditions and stimulating environment.
\\A.M. is supported by the EPSRC First Grant EP/R004730/1 ``Optimal transport and geometric analysis'' and by the ERC Starting Grant  802689 ``CURVATURE''.
\\The authors wish to thank L. Ambrosio for inspiring discussions around the topics of the paper.

\section{Preliminaries}\label{sec:preliminaries}

Throughout the paper $(X,\dist,\meas)$ will be a complete and separable metric measure space with $\supp(\meas)=X$ and $\meas(X)<\infty$. We will denote by $\Borel(X)$ the family of Borel subsets of $X$ and  by $\LIP(X)$ the space of real valued Lipschitz functions over $X$. For any open domain $\Omega\subset X$, $\LIPc(\Omega)$ and $\LIPloc(\Omega)$ will stand for the space of Lipschitz functions with compact support in $\Omega$ and the space of locally Lipschitz functions in $\Omega$.
 Given $u\in\LIPloc(X)$, its slope $\abs{\nabla u}(x)$ is defined as 
\begin{equation*}
\abs{\nabla u}(x):=
\begin{cases*}
\limsup_{y\to x}\frac{\abs{u(x)-u(y)}}{d(x,y)}& \text{if $x$ is not isolated}\\
0&\text{otherwise,}
\end{cases*}
\end{equation*}
moreover we introduce the notation $\Lip(u)$ for the global Lipschitz constant of $u\in\LIP(X)$.
\\For any interval $I\subset \setR$ we will denote by $\AC(I;X)$ the space of absolutely continuous curves $\gamma:I\to X$. For any $\gamma\in\AC(I;X)$,  the metric derivative $\abs{\gamma'}:I\to[0,+\infty]$ defined by
\begin{equation*}
\abs{\gamma'}(t):=\limsup_{s\to t}\frac{\dist(\gamma(s),\gamma(t))}{\abs{t-s}},
\end{equation*} 
provides the following representation of the  the length of $\gamma$:
\begin{equation*}
\mathit{l}(\gamma)=\int_I\abs{\gamma'}(t)\di t.
\end{equation*}

Next we introduce Sobolev functions and Sobolev spaces over $(X,\dist,\meas)$. We refer for instance to \cite{AmbrosioGigliSavare13,AmbrosioColomboDiMarino} for a detailed discussion about this topic.

\begin{definition}[Sobolev spaces and $p$-energy]\label{def:sobolev}
Fix any $1<p<+\infty$. The $p$-Cheeger energy $\Ch_p:L^{p}(X,\meas)\to [0,+\infty]$ is a convex $L^{p}(X,\meas)$-lower semicontinuous functional defined by
\begin{equation}\label{eq:defCheeger}
\Ch_p(f):=\inf\left\lbrace \liminf_{n\to\infty}\frac{1}{p}\int\abs{\nabla f_n}^p\di\meas: f_n\in\LIP(X)\cap L^{p}(X,\meas),\norm{f_n-f}_{L^p}\to 0\right\rbrace. 
\end{equation}
Moreover, we define $W^{1,p}(X,\dist,\meas):=\set{\Ch_p<+\infty}$ and we remark that, when endowed with the norm 
\begin{equation*}
\norm{f}_{W^{1,p}}:=\left(\norm{f}^p_{L^p}+p\Ch_p(f)\right)^{\frac{1}{p}},
\end{equation*}
the Sobolev space $W^{1,p}(X,\dist,\meas)$ is a Banach space.
\end{definition}

By looking at the optimal approximating sequence in \eqref{eq:defCheeger} one can find a minimal object called \textit{minimal weak upper gradient} $\abs{\nabla f}_w$, providing the integral representation
\begin{equation*}
\Ch_p(f)=\frac{1}{p}\int\abs{\nabla f}_w^p\di\meas
\end{equation*} 
for any $f\in W^{1,p}(X,\dist,\meas)$.
We remark that without further regularity assumptions on the metric measure space the minimal weak upper gradient depends also on the integrability exponent $p$; nevertheless we will always omit this dependence in the notation.

\begin{definition}[Local Sobolev spaces]\label{def:sobolevlocal}
Given an open set $\Omega\subset X$, for any $1<p<+\infty$ we will denote by $W^{1,p}_0(\Omega)$ the closure of $\LIPc(\Omega)$ in $W^{1,p}(X,\dist,\meas)$, with respect to the $W^{1,p}$ norm.
\end{definition}

\subsection{Essentially non branching,  $\CD(K,N)$ and $\RCD(K,N)$ metric measure spaces}\label{SS:CDDef}

Denote by 
$$
\Geo(X) : = \{ \gamma \in C([0,1], X):  \sfd(\gamma(s),\gamma(t)) = |s-t| \sfd(\gamma(0),\gamma(1)), \text{ for every } s,t \in [0,1] \}
$$
the space of constant speed geodesics.  The metric space $(X,\sfd)$ is a \emph{geodesic space} if and only if for each $x,y \in X$ 
there exists $\gamma \in \Geo(X)$ so that $\gamma(0) =x, \gamma(1) = y$.
\medskip

We denote with  $\Prob(X)$ the  space of all Borel probability measures over $X$ and with  $\Prob_{2}(X)$ the space of probability measures with finite second moment.
The space $\Prob_{2}(X)$ can be  endowed with the $L^{2}$-Kantorovich-Wasserstein distance  $W_{2}$ defined as follows:  for $\mu_0,\mu_1 \in\Prob_{2}(X)$,  set
\begin{equation}\label{eq:W2def}
  W_2^2(\mu_0,\mu_1) := \inf_{ \pi} \int_{X\times X} \sfd^2(x,y) \, \di \pi(x,y),
\end{equation}
where the infimum is taken over all $\pi \in\Prob(X \times X)$ with $\mu_0$ and $\mu_1$ as the first and the second marginal.
The space $(X,\sfd)$ is geodesic  if and only if the space $(\Prob_2(X), W_2)$ is geodesic. 
\\For any $t\in [0,1]$,  let ${\rm e}_{t}$ be the evaluation map: 
$$
  {\rm e}_{t} : \Geo(X) \to X, \qquad {\rm e}_{t}(\gamma) : = \gamma_{t}.
$$
Any geodesic $(\mu_t)_{t \in [0,1]}$ in $(\Prob_2(X), W_2)$  can be lifted to a measure $\nu \in \Prob(\Geo(X))$, 
so that $({\rm e}_t)_\sharp \, \nu = \mu_t$ for all $t \in [0,1]$. 
\\Given $\mu_{0},\mu_{1} \in\Prob_{2}(X)$, we denote by 
$\Opt(\mu_{0},\mu_{1})$ the space of all $\nu \in \Prob(\Geo(X))$ for which $({\rm e}_0,{\rm e}_1)_\sharp\, \nu$ 
realizes the minimum in \eqref{eq:W2def}. Such a $\nu$ will be called \emph{dynamical optimal plan}. If $(X,\sfd)$ is geodesic, then the set  $\Opt(\mu_{0},\mu_{1})$ is non-empty for any $\mu_0,\mu_1\in\Prob_2(X)$.
\\We will also consider the subspace $\Prob_{2}(X,\sfd,\meas)\subset\Prob_{2}(X)$
formed by all those measures absolutely continuous with respect with $\meas$.
\medskip

A set $G \subset \Geo(X)$ is \emph{a set of non-branching geodesics} if and only if for any $\gamma_{1},\gamma_{2} \in G$, it holds:
$$
\exists \;  \bar t\in (0,1) \text{ such that } \ \forall t \in [0, \bar t\,] \quad  \gamma_{1}(t) = \gamma_{2}(t)   
\quad 
\Longrightarrow 
\quad 
\gamma_{1}(s) = \gamma_{2}(s), \quad \forall s \in [0,1].
$$
In the paper we will mostly consider essentially non-branching spaces, let us recall their definition (introduced by T. Rajala and Sturm \cite{RS2014}). 
\begin{definition}\label{def:ENB}
A metric measure space $(X,\sfd, \meas)$ is \emph{essentially non-branching} (e.n.b. for short) if and only if for any $\mu_{0},\mu_{1} \in\Prob_{2}(X)$,
with $\mu_{0},\mu_{1}$ absolutely continuous with respect to $\meas$, any element of $\Opt(\mu_{0},\mu_{1})$ is concentrated on a set of non-branching geodesics.
\end{definition}
It is clear that if $(X,\sfd)$ is a  smooth Riemannian manifold  then any subset $G \subset \Geo(X)$  is a set of non branching geodesics, in particular any smooth Riemannian manifold is essentially non-branching.
\medskip

In order to formulate curvature properties for $(X,\sfd,\meas)$  we recall the definition of the  distortion coefficients:  for $K\in \R, N\in [1,\infty), \theta \in (0,\infty), t\in [0,1]$, set 
\begin{equation}\label{eq:deftau}
\tau_{K,N}^{(t)}(\theta): = t^{1/N} \sigma_{K,N-1}^{(t)}(\theta)^{(N-1)/N},
\end{equation}
where the $\sigma$-coefficients are defined 
as follows:
given two numbers $K,N\in \R$ with $N\geq0$, we set for $(t,\theta) \in[0,1] \times \R_{+}$, 
\begin{equation}\label{eq:Defsigma}
\sigma_{K,N}^{(t)}(\theta):= 
\begin{cases}
\infty, & \textrm{if}\ K\theta^{2} \geq N\pi^{2}, \crcr
\displaystyle  \frac{\sin(t\theta\sqrt{K/N})}{\sin(\theta\sqrt{K/N})} & \textrm{if}\ 0< K\theta^{2} <  N\pi^{2}, \crcr
t & \textrm{if}\ K \theta^{2}<0 \ \textrm{and}\ N=0, \ \textrm{or  if}\ K \theta^{2}=0,  \crcr
\displaystyle   \frac{\sinh(t\theta\sqrt{-K/N})}{\sinh(\theta\sqrt{-K/N})} & \textrm{if}\ K\theta^{2} \leq 0 \ \textrm{and}\ N>0.
\end{cases}
\end{equation}

\noindent
Let us also recall the definition of the R\'enyi Entropy functional ${\mathcal E}_{N} :\Prob(X) \to [0, \infty]$:
\begin{equation}\label{eq:defEnt}
{\mathcal E}_{N}(\mu)  := \int_{X} \rho^{1-1/N}(x) \,\di\meas,
\end{equation}
where $\mu = \rho \meas + \mu^{s}$ with $\mu^{s}\perp \meas$.
\medskip

The curvature-dimension condition was introduced independently by Lott-Villani \cite{LottVillani} and Sturm \cite{Sturm06I,Sturm06II}, let us recall its definition.

\begin{definition}[$\CD$ condition]\label{def:CD}
Let $K \in \R$ and $N \in [1,\infty)$. A metric measure space  $(X,\sfd,\meas)$ verifies $\CD(K,N)$ if for any two $\mu_{0},\mu_{1} \in \Prob_2(X,\sfd,\meas)$ 
with bounded support there exist $\nu \in \Opt(\mu_{0},\mu_{1})$ and  $\pi\in\Prob(X\times X)$ $W_{2}$-optimal plan such that $\mu_{t}:=(\ee_{t})_{\sharp}\nu \ll \meas$ and for any $N'\geq N, t\in [0,1]$:
\begin{equation}\label{eq:defCD}
{\mathcal E}_{N'}(\mu_{t}) \geq \int \tau_{K,N'}^{(1-t)} (\sfd(x,y)) \rho_{0}^{-1/N'} 
+ \tau_{K,N'}^{(t)} (\sfd(x,y)) \rho_{1}^{-1/N'} \,\di \pi(x,y).
\end{equation}
\end{definition}

It is worth recalling that if $(M,g)$ is a Riemannian manifold of dimension $n$ and 
$h \in C^{2}(M)$ with $h > 0$, then the m.m.s. $(M,\sfd_{g},h \, {\rm Vol}_{g})$ (where $\sfd_{g}$ and ${\rm Vol}_{g}$ denote the Riemannian distance and volume induced by $g$) verifies $\CD(K,N)$ with $N\geq n$ if and only if  (see  \cite[Theorem 1.7]{Sturm06II})
$$
\Ric_{g,h,N} \geq  K g, \qquad \Ric_{g,h,N} : =  \Ric_{g} - (N-n) \frac{\nabla_{g}^{2} h^{\frac{1}{N-n}}}{h^{\frac{1}{N-n}}}.  
$$
In particular if $N = n$ the generalized Ricci tensor $\Ric_{g,h,N}= \Ric_{g}$ makes sense only if $h$ is constant.

The lack of the local-to-global property of the $\CD(K,N)$ condition (for $K/N \neq 0$)
led in 2010 Bacher and Sturm to introduce in \cite{BS10} the reduced curvature-dimension condition, denoted by $\CD^{*}(K,N)$.
The $\CD^{*}(K,N)$ condition asks for the same inequality \eqref{eq:defCD} of $\CD(K,N)$ to hold but
the coefficients $\tau_{K,N}^{(s)}(\sfd(\gamma_{0},\gamma_{1}))$ are replaced by the slightly smaller $\sigma_{K,N}^{(s)}(\sfd(\gamma_{0},\gamma_{1}))$.
\\

Since the $\CD$ condition allows Finsler geometries, in order to single out the ``Riemannian''  structures  Ambrosio-Gigli-Savar\'e \cite{AmbrosioGigliSavare14} introduced the Riemannian curvature dimension condition $\RCD(K,\infty)$ (see also \cite{AGMR12} for the extension to $\sigma$-finite measures and for the present simplification in the axiomatization).
A finite dimensional refinement, coupling the $\CD(K,N)$ condition for $N<+\infty$ with infinitesimal Hilbertianity, has been subsequently proposed in \cite{Gigli1} while the the refined
$\RCD^{*}(K,N)$ condition has been introduced and extensively investigated
in \cite{EKS,AmbrosioMondinoSavare}. We refer to these papers and references therein for a general account
on the synthetic formulation of the latter Riemannian-type Ricci curvature lower bounds; for a survey of results, see the Bourbaki seminar \cite{VilB} and the recent ICM-Proceeding \cite{AmbrosioICM}.
Here we only briefly recall that it is a stable \cite{AmbrosioGigliSavare14,GigliMondinoSavare13} strengthening of the reduced curvature-dimension condition:
a m.m.s. verifies $\RCD^{*}(K,N)$ if and only if it satisfies $\CD^{*}(K,N)$ and is infinitesimally Hilbertian, 
meaning that the Sobolev space $W^{1,2}(X,\meas)$ is a Hilbert space (with the Hilbert structure induced by the Cheeger energy).

To conclude we recall also that recently  Cavalletti and E. Milman \cite{CMi16} proved the equivalence
of $\CD(K,N)$ and $\CD^{*}(K,N)$, together with the local-to-global property for $\CD(K,N)$, in the framework of  essentially non-branching m.m.s. having $\meas(X) < \infty$.
As we will always assume the aforementioned properties to be satisfied by our ambient m.m.s. $(X,\sfd,\meas)$, we will use both formulations with no distinction.
It is worth also mentioning that a m.m.s. verifying $\RCD^{*}(K,N)$ is essentially non-branching (see \cite[Corollary 1.2]{RS2014})
implying also the equivalence of $\RCD^{*}(K,N)$ and  $\RCD(K,N)$.
\\

For all the main results we will assume that the m.m.s. $(X,\sfd,\meas)$ is essentially non-branching and satisfies $\CD(K,N)$ from some $K>0$
with $\supp(\meas) = X$ (or, more strongly, that $(X,\sfd,\meas)$ is a $\RCD(N-1,N)$ space).  It follows that $(X,\sfd)$ is a geodesic and compact metric space with $\meas(X)\in (0,\infty)$. Since $(X,\sfd,\meas)$ is a $\CD(K,N)$ (resp. $\RCD(K,N)$) space if and only if  $(X,\sfd, \alpha \meas)$ is so, without loss of generality we will also assume $\mm(X)=1$.

\subsection{Finite perimeter sets and L\'evy-Gromov isoperimetric inequality}\label{SS:FPSLG}

We now recall the definition  of a finite perimeter set in a metric measure space (see \cite{Am2,MirandaJr} and the more recent \cite{AmbrosioDiMarino}).

\begin{definition}[Perimeter and sets of finite perimeter]
Given a Borel set $E\subset X$ and an open set $A$ the perimeter $\Per(E,A)$ is defined in the following way:
\begin{equation*}
\Per(E,A):=\inf\left\lbrace \liminf_{n\to\infty}\int_A\abs{\nabla u_n}\di\meas: u_n\in\LIPloc(A), u_n\to \chi_E \quad \text{in } L^1_{\loc}(A)\right\rbrace .
\end{equation*}
We say that $E$ has finite perimeter if $\Per(E,X)<+\infty$; in this case, we shall denote $\Per(E):=\Per(E,X)$ to simplify the notation. It can be proved that the set function $A\mapsto\Per(E,A)$ is the restriction to open sets of a finite Borel measure $\Per(E,\cdot)$ defined by
\begin{equation*}
\Per(E,B):=\inf\left\lbrace \Per(E,A): B\subset A,\text{ } A \text{ open}\right\rbrace .
\end{equation*}
\end{definition}

Below we recall the definition of the family of one dimensional model spaces for the curvature dimension condition of parameters $K>0$ and $1<N<+\infty$ (cf.  \cite[Appendix C] {Gro} and \cite{Milman15}).

\begin{definition}[One dimensional model spaces]\label{def:onedimmodel}
For any $K>0$ and for any $1<N<+\infty$ we define the one dimensional model space $(I_{K,N},\dist_{eu}, \meas_{K,N})$ for the curvature dimension condition of parameters $K$ and $N$ by
\begin{equation}\label{eq:defonedim}
I_{K,N}:=\left(0,\pi\sqrt{\frac{N-1}{K}}\right),\quad \meas_{K,N}:=\frac{1}{c_{K,N}}\sin\left({t\sqrt{\frac{K}{N-1}}}\right)^{N-1}\leb^1\res I_{K,N},
\end{equation}
where $\dist$ is the restriction to $I_{K,N}$ of the canonical Euclidean distance over the real line and $c_{K,N}:=\int_{I_{K,N}} \sin\big({t\sqrt{\frac{K}{N-1}}}\big)^{N-1} \, \di \leb^1(t)$ is the normalizing constant.
\\In order to shorten the notation, we set $h_{K,N}(t):=\frac{1}{c_{K,N}}  \sin\left({t\sqrt{\frac{K}{N-1}}}\right)^{N-1}$ for all $t\in I_{K,N}$.
\end{definition}

Let us recall that, for any normalized metric measure space $(X,\dist,\meas)$, the isoperimetric profile $\mathcal{I}_{(X,\dist,\meas)}:[0,1]\to[0,+\infty)$ is defined by
\begin{equation*}
\mathcal{I}_{(X,\dist,\meas)}(v):=\inf\left\lbrace \Per(E):\, E\in\Borel (X),\quad\meas(E)=v \right\rbrace.
\end{equation*}
We will denote by $\mathcal{I}_{K,N}$ the isoperimetric profile of the model space $\left(I_{K,N},\dist_{eu},\meas_{K,N}\right)$.

In \cite{CavallettiMondino17,CavallettiMondino18}, exploiting the so-called localization technique (cf. \cite{klartag}), the following version of the L\'evy-Gromov isoperimetric inequality \cite[Appendix C] {Gro} for metric measure spaces was proven.

\begin{theorem}[L\'evy-Gromov inequality]\label{thm:LevyGromov}
Let $(X,\dist,\meas)$ be an essentially non branching $\CD(K,N)$ metric measure space for some $K>0$ and $1< N<+\infty$. Then, for any Borel set $E\subset X$, it holds
\begin{equation*}
\Per(E)\ge \mathcal{I}_{K,N}(\meas(E)).
\end{equation*}
\end{theorem}

\medskip

We next recall the notion of warped product between metric measure spaces, generalizing the well know Riemannian construction.
\\Given two geodesic metric measure spaces $(B,\dist_B,\meas_B)$ and $(F,\dist_F,\meas_F)$ and a Lipschitz function $f:B\to[0,+\infty)$ one can define a length structure on the product $B\times F$ as follows: for any absolutely continuous curve $\gamma:[0,1]\to B\times F$ with components $(\alpha,\beta)$, define
\begin{equation*}
L(\gamma):=\int_0^1\left(\abs{\alpha'}^2(t)+\left(f\circ\alpha(t)\right)^2\abs{\beta'}^2(t)\right)^{\frac{1}{2}}\di t
\end{equation*} 
and consider the associated pseudo-distance 
\begin{equation*}
\dist((p,x),(q,y)):=\inf\left\lbrace L(\gamma): \gamma(0)=(p,x),\text{ }\gamma(1)=(q,y) \right\rbrace. 
\end{equation*}
The $f$-warped product of $B$ with $F$ is the metric space defined by
\begin{equation*}
B\times_f F:=\left(B\times F/_{\sim},\dist\right),
\end{equation*}
where $(p,x)\sim (q,y)$ if and only if $\dist((p,x),(q,y))=0$. One can also associate a natural measure  and obtain
\begin{equation*}
B\times_f^NF:=\left(B\times_fF,\meas_C\right),\quad \meas_C:=f^N\meas_B\otimes \meas_F,
\end{equation*}
that we will call warped product metric measure space of $(B,\dist_B,\meas_B)$ and $(F,\dist_F,\meas_F)$.
\\

In \cite{CavallettiMondino17,CavallettiMondino18} also the rigidity problem for the L\'evy-Gromov inequality was addressed in the framework of metric measure spaces.  
Before stating the result from  \cite{CavallettiMondino17,CavallettiMondino18}, observe that if $(X,\dist,\meas)$ is an $\RCD(K,N)$ metric measure space for some $K>0$ and $1<N<+\infty$ then the rescaled space $(X,\dist',\meas)$, where $\dist':=\sqrt{\frac{N-1}{K}}\dist$, is an $\RCD(N-1,N)$ space.

\begin{theorem}[Rigidity in L\'evy-Gromov inequality]\label{thm:rigiditylevy}
Let $(X,\dist,\meas)$ be an $\RCD(N-1,N)$ metric measure space for some $N\in[2,+\infty)$ with $\meas(X)=1$. Assume that there exists $\bar{v}\in(0,1)$ such that $\mathcal{I}_{(X,\dist,\meas)}(\bar{v})=\mathcal{I}_{N-1,N}(\bar{v})$.
Then $(X,\dist,\meas)$ is a spherical suspension: there exists an $\RCD(N-2,N-1)$ m.m.s. $(Y,\dist_Y,\meas_y)$ with $\meas_Y(Y)=1$ such that $X$ is isomorphic as a metric measure space to $[0,\pi]\times_{\sin}^{N-1}Y$. Moreover in this case the following hold:
\begin{itemize}
\item[(i)] for any $v\in[0,1]$ we have $\mathcal{I}_{(X,\dist,\meas)}(v)=\mathcal{I}_{N-1,N}(v)$;
\item[(ii)] for any $v\in[0,1]$ there exists a Borel set $A\subset X$ with $\meas(A)=v$ and such that
\begin{equation*}
\Per(A)=\mathcal{I}_{(X,\dist,\meas)}(v)=\mathcal{I}_{N-1,N}(v);
\end{equation*}
\item[(iii)] for any Borel set $A\subset X$ such that $\meas(A)=v$, we have $\Per(A)=\mathcal{I}_{N-1,N}(v)$ if and only if
\begin{equation*}
\meas(A\setminus\{(t,y)\in[0,\pi]\times Y: t\in[0,r(v)]\})=0
\end{equation*}
or
\begin{equation*}
\meas(A\setminus\{(t,y)\in[0,\pi]\times Y: t\in[\pi-r(v),\pi]\})=0, 
\end{equation*}
where $r(v)\in(0,\pi)$ is such that $c_N\int_0^{r(v)}\sin^{N-1}(t)\di t=v$ (with $c_N$ normalization constant).
\end{itemize}
\end{theorem}

\subsection{$\BV$ functions and coarea formula in m.m.s.}

As for the classical Euclidean case, in metric measure spaces one can introduce not only a notion of finite perimeter set but also a notion of function of bounded variation. We refer again to \cite{MirandaJr} and \cite{AmbrosioDiMarino} for more details about the topic.

\begin{definition}
A function $f\in L^1(X,\meas)$ is said to belong to the space $\BV_*(X,\dist,\meas)$ if there exists a sequence of locally Lipschitz functions $f_i$ converging to $f$ in $L^1(X,\meas)$ such that
\begin{equation*}
\limsup_{i\to\infty}\int_X\abs{\nabla f_i}\di \meas<+\infty.
\end{equation*}  
By localizing this construction one can define 
\begin{equation*}
\abs{Df}_*(A):=\inf\left\lbrace \liminf_{i\to\infty}\int_A\abs{\nabla f_i}\di \meas: f_i\in\LIPloc(A), f_i\to f \text{ in } L^1(A)\right\rbrace  
\end{equation*}
for any open $A\subset X$. In \cite{AmbrosioDiMarino} it is proven that this set function is the restriction to open sets of a finite Borel measure that we call \emph{total variation of $f$}.
\end{definition}
For any Lipschitz function $f:X\to\setR$ it is easy to check that $f\in \BV_{*}(X,\dist,\meas)$ and $\abs{Df}_*\le\abs{\nabla f}\meas$. In the following we will denote by $\abs{\nabla f}_1$ the density of $\abs{Df}_{*}$ with respect to  $\meas$.
With a slight abuse of notation motivated by simplicity, we are going to use the same symbol  $\abs{\nabla f}_1$ to denote the equivalence class (under $\mm$-a.e. equality) and  a Borel representative.
\\ The following result is a simplified version of \cite[Proposition 4.2]{AmbrosioPinamontiSpeight}.

\begin{proposition}\label{prop:localityof1}
Let $f\in\LIP(X)$. Then $\abs{\nabla f}_1(x)=0$ for $\meas$-a.e. $x\in\{f=0\}$.
\end{proposition}

Let us point out that on general metric measure spaces there is no reason to expect any identification, in the almost everywhere sense, between $\abs{\nabla f}_1$, $\abs{\nabla f}$ and the $p$-minimal weak upper gradient $\abs{\nabla f}_w$, for some $p>1$, of a Lipschitz function $f$.\\
The identification result stated below is a consequence of the seminal work \cite{Cheeger} concerning Lipschitz functions on metric measure spaces satisfying doubling and Poincar\'e inequalities and of the identification result for $p$-minimal weak upper gradients obtained in \cite{GigliHan14} for proper $\RCD(K,\infty)$ spaces. In particular in \cite[Remark 3.5]{GigliHan14} it is observed that, for any Lipschitz function $f$, the density of $\abs{Df}_{*}$ with respect to $\meas$ coincides with the $p$-minimal weak upper gradient for any $p>1$. In \cite{Cheeger} instead it is proved that the slope coincides with the $p$-minimal weak upper gradient $\meas$-almost everywhere. Combining these two ingredients we get:

\begin{proposition}\label{prop:identification}
Let $(X,\dist,\meas)$ be an $\RCD(K,N)$ space, for some $K\in \R, N\in (1,\infty)$. Then for any $f\in\LIP(X)$ one has that $\abs{Df}_*=\abs{\nabla f}\meas$.
\end{proposition}

The following coarea formula for functions of bounded variation on metric measure spaces is taken from \cite[Remark 4.3]{MirandaJr}. It will play a key role in the rest of the paper.

\begin{theorem}[Coarea formula]\label{thm:coarea}
Let $\Omega\subset X$ be an open domain in a m.m.s. $(X,\dist,\meas)$.
\\Let $v:\Omega\to[0,+\infty)$ be a continuous function belonging to $\BV_*(\Omega,\dist,\meas)$.
Then for any Borel function $f:\Omega\to[0,+\infty)$ it holds 
\begin{equation}\label{eq:coarea}
\int_{\{s\le v<t\}} f\di\abs{Dv}_*=\int_s^{t}\left(\int f\di\Per(\{v>r\})\right)\di r, \quad \forall s\in [0,t].
\end{equation}
\end{theorem}
Combining \autoref{prop:identification} and \autoref{thm:coarea} we obtain the following.

\begin{corollary}\label{cor:coareaRCD}
Let $(X,\dist,\meas)$ be an $\RCD(K,N)$ space, for some $K\in \R, N\in (1,\infty)$.
\\ Let $\Omega\subset X$ be an open domain and $v:\Omega\to[0,+\infty)$ be Lipschitz. Then, for any Borel function $f:\Omega\to[0,+\infty)$ it holds that
\begin{equation}\label{eq:coarea1}
\int_{\{s\le v<t\}} f\abs{\nabla v}\di \meas=\int_s^{t}\left(\int f\di\Per(\{v>r\})\right)\di r,  \quad \forall s\in [0,t].
\end{equation}
\end{corollary}
The following result will be useful when dealing with the almost rigidity case in the spectral gap inequality. We refer to \cite[Chapter 27]{Vil09} (see also \cite[Section 3.5]{GigliMondinoSavare13}) for the relevant background about measured Gromov-Hausdorff convergence.

\begin{proposition}\label{prop:lscprofiles}
Fix  $K>0$ and $N\in (1,\infty)$. Let $\left((X_n,\dist_n,\meas_n)\right)_{n\in\setN}$ be a sequence of normalized $\RCD(K,N)$ spaces converging to $(X,\dist,\meas)$  in the measured Gromov-Hausdorff sense. 
\\ Denote by $\mathcal{I}_n$ (resp. $\mathcal{I}$) the isoperimetric profile of $(X_n,\dist_n,\meas_n)$ (resp. of  $(X,\dist,\meas)$). 
 \\Then, for any $t\in [0,1]$ and for any sequence $(t_n)_{n\in\setN}$ with $t_n\to t$, it holds that
\begin{equation}\label{eq:lscprofile}
\mathcal{I}(t)\le\liminf_{n\to\infty}\mathcal{I}_n(t_n).
\end{equation}
\end{proposition}

\begin{proof}
We refer to \cite{GigliMondinoSavare13,AmbrosioHonda} for the basic definitions and statements about convergence of functions defined over mGH-converging sequences of metric measure spaces.
\\First of all note that in order to prove \eqref{eq:lscprofile}, without loss of generality we can assume that $\sup_{n\in\setN}\mathcal{I}_n(t_n)<+\infty$.
\\For any $n\in\setN$ let $E_n\subset X_n$ be a Borel set such that $\Per_n(E_n)=\mathcal{I}_n(t_n)$, whose  existence  follows as in the Euclidean case from standard lower semicontinuity and compactness arguments.  
\\The sequence of the corresponding characteristic functions $(\chi_{E_n})_{n\in\setN}$ satisfies the assumption of \cite[Proposition 7.5]{AmbrosioHonda}, i.e.
\begin{equation*}
\sup_{n\in\setN}\left\lbrace\norm{\chi_{E_n}}_{L^1(\meas_n)}+\abs{D\chi_{E_n}}(X_n) \right\rbrace=\sup_{n\in\setN}\left\lbrace t_n+\mathcal{I}_n(t_n) \right\rbrace<+\infty.  
\end{equation*}
It follows from \cite[Proposition 7.5]{AmbrosioHonda} that, up to extracting a subsequence which we do not relabel, $\left(\chi_{E_n}\right)_{n\in\setN}$ strongly $L^{1}$-converges  to a function $f\in L^1(X,\meas)$ (see \cite[Section 3]{AmbrosioHonda}). In particular we can say that
\begin{equation}
\norm{f}_{L^1(\meas)}=\lim_{n\to\infty}\norm{\chi_{E_n}}_{L^1(\meas_n)}=\lim_{n\to\infty}t_n=t.
\end{equation}
We now claim that $f$ is the indicator function of a Borel set $E\subset X$, with $\meas(E)=t$. To this aim call $g_n:=\chi_{E_n}(1-\chi_{E_n})$ and observe that $(g_n)_{n\in \setN}$ strongly $L^1$-converges to $g:=f(1-f)$ thanks to \cite[Proposition 3.3]{AmbrosioHonda}. Thus $g=0$, since $g_n=0$ for any $n\in\setN$ and therefore $g$ is the indicator function of a Borel set, as claimed. 
\\We can now apply  \cite[Theorem 8.1]{AmbrosioHonda} to get  the Mosco convergence of the $\BV$ energies and   conclude that
\begin{equation*}
\Per(E)\le\liminf_{n\to\infty}\Per_n(E_n)=\liminf_{n\to\infty}\mathcal{I}_n(t_n).
\end{equation*}
The lower semicontinuity for the isoperimetric profiles \eqref{eq:lscprofile} easily follows, since $E$ is an admissible competitor in the definition of $\mathcal{I}(t)$.
\end{proof}

\section{Polya-Szego inequality}\label{sec:polyaszego}

The working assumption of this section, unless otherwise stated, is that $(X,\dist,\meas)$ is an essentially non branching $\CD(K,N)$ space for some $K>0$, $N\in (1,+\infty)$, with $\meas(X)=1$ and $\supp(\meas)=X$. 

\begin{definition}[Distribution function]\label{def:distributionfunction}
Given an open domain $\Omega\subset X$ and a non-negative Borel function $u:\Omega\to[0,+\infty)$ we define its distribution function $\mu:[0,+\infty)\to[0,\meas(\Omega)]$ by 
\begin{equation}\label{eq:defdistr}
\mu(t):=\meas(\{u> t\}).
\end{equation}

\end{definition}

\begin{remark}
Suppose that $u$ is such that $\meas(\{u=t\})=0$ for any $0<t<+\infty$. Then it makes no difference to consider closed superlevel sets or open superlevel sets in \eqref{eq:defdistr}.
\end{remark}

It is not difficult to check that the distribution function $\mu$ is non increasing and left-continuous. If moreover $u$ is continuous, then  $\mu$ is strictly decreasing.
We let $u^{\#}$ be the generalized inverse of $\mu$, defined in the following way:
\begin{equation*}
u^{\#}(s):=
\begin{cases*}
{\ess}\sup u & \text{if $s=0$},\\
\inf\left\lbrace t:\mu(t)<s \right\rbrace &\text{if $s>0$}. 

\end{cases*}
\end{equation*}

\begin{definition}[Rearrangement on one dimensional model spaces] 
Fix any $K>0$, $1<N<+\infty$, and let  $\left(I_{K,N},\dist,\meas_{K,N}\right)$ be the one-dimensional model space defined in \eqref{eq:defonedim}.  Let $\Omega\subset X$ be an open subset and consider  $[0,r]\subset I_{K,N}$ such that $\mm_{K,N}([0,r])=\mm(\Omega)$.
\\ For any  Borel function $u:\Omega\to [0,+\infty)$, the \emph{monotone rearrangement} $u^*_{K,N}:[0,r]\to \R^{+}$  is defined by
\begin{equation*}
u^*_{K,N}(x):=u^{\#}(\meas_{K,N}([0,x])), \quad \forall x\in [0,r].
\end{equation*}
For simplicity of notation we will often write $u^{*}$ in place of $u^*_{K,N}$.
\end{definition}

\begin{remark}\label{rem:RearSignChang}
For simplicity of notation, throughout the paper we will consider monotone rearrangements of  \emph{non-negative} functions. Nevertheless, for an arbitrary Borel function $u:\Omega\to(-\infty+\infty)$ the analogous statements hold by setting $u^*$ the monotone rearrangement of $\abs{u}$. 
\end{remark}

In the next proposition we collect some  useful properties of the monotone rearrangement,  whose proof in the Euclidean setting can be found for instance in \cite[Chapter 1]{Kesavan} and can be adapted with minor modifications to our framework.

\begin{proposition}\label{prop:equimeasurability}
Let $(X,\dist,\meas)$ with $\mm(X)=1$ be an essentially non branching $\CD(K,N)$ space for some $K>0$, $N\in(1,+\infty)$. Let $\Omega\subset X$ be an open subset and consider  $[0,r]\subset I_{K,N}$ such that $\mm_{K,N}([0,r])=\mm(\Omega)$.
Let $u:\Omega\to[0,+\infty)$ be Borel and let $u^*:[0,r]\to[0,+\infty)$ be its monotone rearrangement.
\\ Then $u$ and $u^*$ have the same distribution function (we will often say that they are equimeasurable). 
Moreover,
\begin{equation}\label{eq:equalLpnorm}
\norm{u}_{L^p(\Omega,\meas)}=\norm{u^*}_{L^p([0,r],\meas_{K,N})}, \quad \forall 1\le p<+\infty,
\end{equation}
and the monotone rearrangement operator $L^{p}(\Omega,\meas)\ni u \mapsto u^{*} \in L^{p}([0,r],\meas_{K,N})$ is continuous.
\end{proposition}

Motivated by the working assumptions of \autoref{lemma:derivativedistribution}, we state and prove the following general result about approximation via functions with non vanishing minimal weak upper gradient.

\begin{lemma}[Approximation with non vanishing gradients]\label{lemma:approximationwithnonvanishing}
Let $(X,\dist,\meas)$ be a locally compact geodesic metric measure space  and let  $\Omega\subset X$ be an open subset with $\meas(\Omega)<+\infty$.  
\\Then for any non-negative $u\in\LIPc(\Omega)$ with $\int\abs{\nabla u}^p\di \meas<\infty$, there exists a sequence $(u_n)_{n\in\setN}$  with $u_n\in \LIPc(\Omega)$ non-negative, $\abs{\nabla u_n}\neq 0$ $\meas$-a.e. on $\set{u_n>0}$ for any $n\in\setN$ and such that $u_n\to u$ in $W^{1,p}(X,\meas)$. In particular $\int\abs{\nabla u_n}^p\di \meas\to\int\abs{\nabla u}^p\di \meas$ as $n\to\infty$.
\end{lemma}

\begin{proof}
It is straightforward to check that there exists a sequence $\left(\epsilon_n\right)_{n\in\setN}$ monotonically converging to $0$ from above such that $\meas(\{\abs{\nabla u}=\epsilon_n\})=0$ for any $n\in\setN$.
\\Choose an open set $\Omega'$ containing the support of $u$ and compactly contained in $\Omega$.
Let $v:\Omega\to[0,+\infty)$ be the distance function from the complementary of $\Omega'$ in $X$, namely
\begin{equation*}
v(x):=\Dist(x,X\setminus\Omega')\quad\text{for any $x\in\Omega$}.
\end{equation*}
Observe that $v\in \LIPc(\Omega)$, moreover
\begin{equation}\label{eq:gradientdistance}
\abs{\nabla v}(x)=1\quad\text{for any $x\in\Omega'$}.
\end{equation}
Indeed it suffices to observe that the restriction of $v$ to any geodesic connecting $x$ with $y\in X\setminus\Omega'$ such that $v(x)=\dist(x,y)$ has slope equal to $1$ at $x$.

Next we introduce the approximating sequence $u_n:=u+\epsilon_nv$ and we claim that it has the desired properties. Indeed, if $u\in \LIPc(\Omega)$ is non-negative, then also $u_n\in\LIPc(\Omega)$ is so. From the inequality
\begin{equation*}
\abs{\nabla(u+\epsilon_nv)}\ge\abs{\abs{\nabla u}-\epsilon_n\abs{\nabla v}}
\end{equation*}
and from \eqref{eq:gradientdistance} it follows that $\{\abs{\nabla{u_n}}=0\}\cap \set{u_n>0}\subset\{\abs{\nabla{u}}=\epsilon_n\}$. Since the $\epsilon_n$ are chosen in such a way that $\meas(\{\abs{\nabla u}=\epsilon_n\})=0$, we infer that  $\meas(\{\abs{\nabla u_n}=0\}\cap\set{u_n>0})=0$ .
\\Clearly $u_n$ converge uniformly to $u$ as $n\to\infty$, granting in particular that $u_n\to u$ in $L^{p}(\Omega,\meas)$.\\
At the same time it holds that
\begin{equation*}
\abs{\nabla (u_n-u)}=\eps_n\abs{\nabla v}.
\end{equation*}
Therefore
\begin{equation*}
\int_{\Omega}\abs{\nabla (u_n-u)}^p\di\meas=\eps_n^p\int_{\Omega}\abs{\nabla v}^p\di\meas\to 0,
\end{equation*}
yielding $u_n\to u$ in $W^{1,p}(X,\dist,\meas)$.
\end{proof}

\begin{corollary}\label{cor:approximationnonvanishing1}
Let $(X,\dist,\meas)$ be a geodesic metric measure space verifying locally doubling and a weak $1$-$1$ Poincar\'e inequality  and let  $\Omega\subset X$ be an open subset with  $\meas(\Omega)<+\infty$.
\\Fix any $1<p<+\infty$ and let $u,u_n\in \LIPc(\Omega)$ be as in the statement and the proof of \autoref{lemma:approximationwithnonvanishing}. Then, for any $n\in\setN$, it holds that $\abs{\nabla u_n}_1(x)\neq 0$ for  $\meas$-a.e. $x\in\set{u_n>0}$. 
\end{corollary}

\begin{proof}
One of the properties of the approximating sequence in \autoref{lemma:approximationwithnonvanishing} is that $\abs{\nabla u_n}(x)\neq 0$ for $\meas$-a.e. $x\in\set{u_n>0}$. The desired conclusion follows from \cite{AmbrosioPinamontiSpeight} where it is proved that, under the locally doubling and Poincar\'e assumption, there exists $c>0$ such that
\begin{equation*}
\abs{\nabla f}(x)\le c\abs{\nabla f}_1(x) \quad \text{for $\meas$-a.e. $x\in X$},
\end{equation*}
for any function $f\in \LIPloc(X)$.
\end{proof}

\begin{remark}\label{rem:CDPoincare}
Since any essentially non branching $\CD(K,N)$ metric measure space is (locally) doubling and verifies a weak $1$-$1$ Poincar\'e inequality (see \cite{VonRenesse08}), \autoref{cor:approximationnonvanishing1} applies to the case of our interest. 
\end{remark}

In \autoref{prop:liptolip} below we extend to the non smooth setting \cite[Theorem 2.3.2]{Kesavan}; the key idea is to replace the Euclidean isoperimetric inequality with the L\'evy-Gromov isoperimetric inequality \autoref{thm:LevyGromov}.

\begin{proposition}[Lipschitz to Lipschitz property of the rearrangement]\label{prop:liptolip}
Let $(X,\dist,\meas)$ with $\mm(X)=1$ be an essentially non branching $\CD(K,N)$ space for some $K>0$, $N\in(1,+\infty)$. Let $\Omega\subset X$ be an open subset and consider  $[0,r]\subset I_{K,N}$ such that $\mm_{K,N}([0,r])=\mm(\Omega)$. 
\\Let  $u\in \LIP(\Omega)$ be non-negative with Lipschitz constant $L\ge0$ and  assume that $\abs{\nabla u}_1(x)\neq 0$ for $\meas$-a.e. $x\in\set{u>0}$. 
\\Then $u^*:[0,r]\to[0,\infty)$ is $L$-Lipschitz as well.
\end{proposition}

\begin{proof}
Let $\mu$ be the distribution function associated to $u$ and denote by $M:=\sup{u}<+\infty$.
Observe that our assumptions grant continuity and strict monotonicity of $\mu$. Therefore for any $s,k\ge 0$ such that $s+k\le\meas(\Omega)$ we can find $0\le t-h\le t\le M$ in such a way that $\mu(t-h)=s+k$ and $\mu(t)=s$.
Taking into account the assumption that $u$ is $L$-Lipschitz we can say that
\begin{equation}\label{eq:estimatedistribution}
\int_{\{t-h\le u\le t\}}\abs{\nabla u}_1\di \meas\le L\left(\mu(t-h)-\mu(t)\right).
\end{equation}
On the other hand, an application of the coarea formula \eqref{eq:coarea} yields 
\begin{equation}\label{eq:intliptolip}
\int_{\{t-h\le u\le t\}}\abs{\nabla u}_1\di \meas=\int_{t-h}^t\Per(\{u\ge r\})\di r.
\end{equation}
Applying the L\'evy-Gromov isoperimetric inequality \autoref{thm:LevyGromov} we can estimate the right hand side of \eqref{eq:intliptolip} in the following way:
\begin{equation}\label{eq:estimateisoperimetric}
\int_{t-h}^t\Per(\{u\ge r\})\di r\ge \int_{t-h}^t\mathcal{I}_{K,N}(\mu(r))\di r.
\end{equation}
Recalling that the model isoperimetric profile $\mathcal{I}_{K,N}$ is continuous and that $\mu$ is continuous, combining \eqref{eq:estimatedistribution} with \eqref{eq:estimateisoperimetric} and eventually applying the mean value theorem we get
\begin{equation}\label{eq:contestimate}
Lk\ge \int_{t-h}^t\mathcal{I}_{K,N}(\mu(r))\di r=h\mathcal{I}_{K,N}(\mu(\xi_{t-h}^t)),
\end{equation}
for some $t-h\le\xi_{t-h}^t\le t$. Calling $u^{\#}$ the inverse of the distribution function, the estimate  \eqref{eq:contestimate} can be rewritten as
\begin{equation}\label{eq:loclipest}
\left(u^{\#}(s)-u^{\#}(s+k)\right)\mathcal{I}_{K,N}(\mu(\xi_{t-h}^t))\le Lk. 
\end{equation}
Since $\mathcal{I}_{K,N}$ is strictly positive on $(0,1)$, it follows from \eqref{eq:loclipest} that $u^{\#}$ is locally Lipschitz. Moreover, at any differentiability point $s$ of $u^{\#}$ (which in particular form a set of full $\leb^1$-measure on $(0,1)$), it holds
\begin{equation}\label{eq:estimatederivative}
-\didi{s}u^{\#}(s)\le \frac{L}{\mathcal{I}_{K,N}(s)}.
\end{equation}
Let $r:[0,1]\to I_{K,N}$ be such that $r(\meas_{K,N}([0,x]))=x$ for any $x\in I_{K,N}$. Differentiating in $t$ the identity
\begin{equation*}
\int_0^{r(t)}h_{K,N}(s)\di s=t,
\end{equation*}
we obtain that $1=\didi{t}r(t)h_{K,N}(r(t))$ and, since we know that $\mathcal{I}_{K,N}(s)=h_{K,N}(r(s))$, it follows that 
\begin{equation}\label{eq:expressionderivative}
\didi{t}r(t)=\frac{1}{\mathcal{I}_{K,N}(t)}.
\end{equation}
By definition of the rearrangement $u^*$,  for any $x\in I_{K,N}$ it holds that $u^*(x)=u^{\#}(\meas_{K,N}([0,x]))$. Combining the last identity with \eqref{eq:estimatederivative} and \eqref{eq:expressionderivative} we can estimate for $x\leq y$
\begin{align*}
0\leq u^*(x)-u^*(y)&=u^{\#}(\meas_{K,N}([0,x]))-u^{\#}(\meas_{K,N}([0,y]))\\
&=\int_{\meas_{K,N}([0,x])}^{\meas_{K,N}([0,y])}-\didi{s}u^{\#}(s)\di s\\
&\le\int_{\meas_{K,N}([0,x])}^{\meas_{K,N}([0,y])}L\didi{s}r(s)\di s\\
&=Lr(\meas_{K,N}([0,y]))-Lr(\meas_{K,N}([0,x]))=Ly-Lx,
\end{align*}
which gives the $L$-Lipschitz continuity of the monotone rearrangement $u^*$.
\end{proof}

The next lemma should be compared with  \cite{Kesavan}, dealing with the case of smooth functions in Euclidean domains.

\begin{lemma}[Derivative of the distribution function]\label{lemma:derivativedistribution}
Let $(X,\dist,\meas)$ be a metric measure space and let $\Omega\subset X$ be an open subset with $\meas(\Omega)<+\infty$.
Assume that $u\in \LIPloc(\Omega)$ is non-negative and $\abs{\nabla u}_1(x)\neq0$ for $\meas$-a.e. $x\in\set{u>0}$. Then its distribution function $\mu:[0, +\infty)\to [0, \meas(\Omega)]$, defined in \eqref{eq:defdistr}, is absolutely continuous. Moreover it holds
\begin{equation}\label{eq:derivativedistributionfunction}
\mu'(t)=-\int\frac{1}{\abs{\nabla u}_1}\di \Per(\{u>t\}) \quad \text{for $\leb^1$-a.e. $t$},
\end{equation}
where the quantity $1/\abs{\nabla u}_1$ is defined to be $0$ whenever $\abs{\nabla u}_1=0$.
\end{lemma}

\begin{proof}
Fix any $\epsilon>0$ and define
\begin{equation*}
f_{\epsilon}(x):=\frac{\abs{\nabla u(x)}_1}{\abs{\nabla u(x)}_1^2+\epsilon}.
\end{equation*}
Fixing $t\ge 0$ and $h>0$, an application of the coarea formula \eqref{eq:coarea} with $f=f_{\epsilon}$ yields to
\begin{equation}\label{eq:intermderiv}
\int_{\{t\le u\le t+h\}}\frac{\abs{\nabla u}_1^2}{\abs{\nabla u}_1^2+\epsilon}\di \meas=\int_{t}^{t+h}\left(\int\frac{\abs{\nabla u}_1}{\abs{\nabla u}_1^2+\epsilon}\di \Per(\{u> r\})\right)\di r.
\end{equation}
Now we pass to the limit as $\epsilon \to 0$ both at the right hand side and at the left hand side in \eqref{eq:intermderiv}. The assumption that $\abs{\nabla u}_1\neq0$ $\meas$-a.e. grants that the integrand at the left hand side monotonically converges $\meas$-a.e. to $1$. Thus an application of the monotone convergence theorem yields that
\begin{equation}\label{eq:intermediate2}
\int_{\{t\le u\le t+h\}}\frac{\abs{\nabla u}_1^2}{\abs{\nabla u}_1^2+\epsilon}\di \meas\to \mu(t)-\mu(t+h)\quad \text{as $\epsilon \to 0$}.
\end{equation}

With the above mentioned convention about the value of $1/\abs{\nabla u}_1$ at points where $\abs{\nabla u}_1=0$, applying the monotone convergence theorem twice at the right hand side of \eqref{eq:intermderiv}, we get
\begin{equation}\label{eq:int3derivativedistri}
\int_{t}^{t+h}\left(\int\frac{\abs{\nabla u}_1}{\abs{\nabla u}_1^2+\epsilon}\di \Per(\{u> r\})\right)\di r\to \int_{t}^{t+h}\left(\int\frac{1}{\abs{\nabla u}_1}\di \Per(\{u>r\})\right)\di r
\end{equation}
as $\epsilon$ goes to $0$.
Combining \eqref{eq:intermderiv}, \eqref{eq:intermediate2} and \eqref{eq:int3derivativedistri}, we get
\begin{equation*}
\mu(t)-\mu(t+h)=\int_{t}^{t+h}\left(\int\frac{1}{\abs{\nabla u}_1}\di \Per(\{u>r\})\right)\di r.
\end{equation*}
It follows that the distribution function is absolutely continuous and therefore differentiable at almost all points with derivative given by the explicit expression \eqref{eq:derivativedistributionfunction}.
\end{proof}

Before proceeding to the statement and the proof of the Polya-Szego inequality we need an identification result between slopes and $1$-minimal weak upper gradients in the simplified setting of the model weighted interval $I_{K,N}$.  In this setting, for any $p\ge 1$, we say that $u\in W^{1,p}(I,\dist,\meas_{K,N})$ if the distributional derivative of $u$ (defined through integration by parts) is in $L^p(I,\meas_{K,N})$. We refer to \cite[Section 8]{AmbrosioDiMarino} for an account about $W^{1,1}$ spaces on metric measure spaces.

\begin{lemma}\label{prop:identificationonedim}
Let $I\subset I_{K,N}$ be  a sub-interval and let $1<p<+\infty$. Let $f\in W^{1,p}((I,\dist,\meas_{K,N}))$ be  monotone.
Then $f\in W^{1,1} ((I,\dist,\meas_{K,N}))$ and it holds
\begin{equation}\label{eq:N1n}
\abs{\nabla f}_1(x)=\abs{f'}(x)=\abs{\nabla f}(x) \quad\text{for $\leb^1$-a.e. $x\in I$}.
\end{equation}
\end{lemma}

\begin{proof}
The fact that $f\in W^{1,1} ((I,\dist,\meas_{K,N}))$ follows directly by H\"older inequality, since $\meas_{K,N}(I)\leq 1$. Since $\mm_{K,N}=h_{K,N} \leb^1$ with $h_{K,N}$ locally bounded away from $0$ out of the two end-points of $I_{K,N}$, it follows that $f$ is locally absolutely continuous in the interior of $I_{K,N}$. In particular it is differentiable  $\leb^1$-a.e. and $\abs{\nabla f}(x)=\abs{f'}(x)$ at every differentiability point $x$. We are thus left to show the first equality in \eqref{eq:N1n}.
\\Note that the assumptions ensure that $f$ is invertible onto its image, up to a countable subset of $f(I)$. The coarea formula in the 1-dimensional case reads as
\begin{equation}\label{eq:coarea1f}
\int_I (\varphi \cdot h_{K,N})  \abs{\nabla f}_{1} \di\leb^{1} =\int_I \varphi \abs{\nabla f}_{1}  \di \meas_{K,N}=\int_{f(I)} (\varphi \cdot h_{K,N}) (f^{-1}(r)) \di r,  \quad \forall \varphi\in C_{c}(I_{K,N}).
\end{equation}
On the other hand, the change of variable formula via a monotone absolutely continuous function gives
\begin{equation}\label{eq:coarea1f2}
\int_I (\varphi \cdot h_{K,N})  |f'| \di\leb^{1} =\int_{f(I)} (\varphi \cdot h_{K,N}) (f^{-1}(r)) \di r,  \quad \forall \varphi\in C_{c}(I_{K,N}).
\end{equation}
The combination of \eqref{eq:coarea1f} with \eqref{eq:coarea1f2} then gives the first equality in \eqref{eq:N1n}.
\end{proof}

The following statement should be compared with \cite{Kesavan}, where the study of the monotone rearrangement on domains of $\setR^n$ is performed.

\begin{proposition}\label{thm:polyaszego}
Let $(X,\dist,\meas)$ be an essentially non branching $\CD(K,N)$ space for some $K>0$, $N\in (1,+\infty)$.  Let $\Omega\subset X$ be an open subset and consider  $[0,r]\subset I_{K,N}$ such that $\mm_{K,N}([0,r])=\mm(\Omega)$.  
\\Let $u\in \LIP(\Omega)$ be non-negative and assume that $\abs{\nabla u}_1(x)\neq 0$ for $\meas$-a.e. $x\in\set{u>0}$. 
\\Then $u^*\in \LIP([0,r])$  and  for any $1<p<+\infty$ it holds 
\begin{equation}\label{eq:polyaszego-1}
\int_{\Omega}\abs{\nabla u}_{1}^p\di \meas\ge\int_{0}^{r}\abs{\nabla u^*}_{1}^p\di \meas_{K,N}.
\end{equation}
In particular, it follows that 
\begin{equation}\label{eq:polyaszego}
\int_{\Omega}\abs{\nabla u}^p\di \meas\ge\int_{0}^{r}\abs{\nabla u^*}^p\di \meas_{K,N}.
\end{equation}
\end{proposition}

\begin{proof}
Denote by $M:=\sup u$.
Since $u$ is Lipschitz we are in position to apply \autoref{prop:liptolip}, which grants that the monotone rearrangement $u^*$ is still Lipschitz.
\\Introduce the functions $\varphi,\psi:[0,M]\to[0,+\infty)$ defined by
\begin{equation*}
\varphi(t):=\int_{\{u>t\}}\abs{\nabla u}_1^p\di \meas,\quad \psi(t):=\int_{\{u>t\}}\abs{\nabla u}_1\di \meas.
\end{equation*}
An application of the coarea formula \autoref{thm:coarea} yields that $\varphi$ and $\psi$ are absolutely continuous and therefore $\leb^1$-a.e. differentiable with derivatives given $\leb^1$-a.e. by the expressions
\begin{equation*}
\varphi'(t)=-\int\abs{\nabla u}_1^{p-1}\di\Per(\{u> t\})\text{ and } \psi'(t)=-\Per(\{u> t \}),
\end{equation*}
respectively.
An application of H\"older's inequality yields that for any $0\le t-h\le t\le M$
\begin{equation}\label{eq:holderestimate}
\int_{\{t-h< u\le t \}}\abs{\nabla u}_1\di \meas\le \left(\int_{\{t-h< u\le t\}}\abs{\nabla u}_1^p\di \meas\right)^{\frac{1}{p}}\left(\mu(t-h)-\mu(t)\right)^{\frac{p-1}{p}},
\end{equation}
where $\mu$ denotes the distribution function associated to $u$.
It follows from the discussion above and from \autoref{lemma:derivativedistribution} that $\leb^1$-a.e. point $t\in(0,M)$ is a differentiability point of both $\mu,\varphi$ and $\psi$. In view of \eqref{eq:holderestimate}, at any such point it holds that 
\begin{equation}\label{eq:estimatepolya}
-\psi'(t)\le \left(-\varphi'(t)\right)^{\frac{1}{p}}\left(-\mu'(t)\right)^{\frac{p-1}{p}}.
\end{equation}
Applying the L\'evy-Gromov inequality \autoref{thm:LevyGromov} we obtain that $\Per(\{u> t\})\ge \mathcal{I}_{K,N}(\mu(t))$. Therefore, taking into account the strict monotonicity of $\mu$, \eqref{eq:estimatepolya} turns into
\begin{equation}\label{eq:levyapplied}
-\varphi'(t)\ge\frac{\left(\mathcal{I}_{K,N}(\mu(t))\right)^p}{\left(-\mu'(t)\right)^{p-1}}\quad \text{for $\leb^1$-a.e. $t$}.
\end{equation}
Thus
\begin{equation}\label{eq:boundfrombelow}
\int_{\Omega}\abs{\nabla u}_1^p\di \meas=\int_0^M-\varphi'(t)\di t\ge\int_0^M\frac{\left(\mathcal{I}_{K,N}(\mu(t))\right)^p}{\left(-\mu'(t)\right)^{p-1}}\di t.
\end{equation}
It follows from the very definition of the monotone rearrangement and from the properties of the model isoperimetric profile that $\Per(\{u^*> t\})=\mathcal{I}_{K,N}(\mu(t))$ (recall that $u$ and $u^*$ have the same distribution function). Moreover, since we already know that $u^*$ is Lipschitz, we are in  position to apply \autoref{lemma:derivativedistribution} to conclude (taking also into account \autoref{prop:identificationonedim}) that

\begin{equation}\label{eq:derdistr}
-\mu'(t)=\frac{\Per(\{u^*> t\})}{\abs{(u^*)'((u^*)^{-1}(t))}} \quad\text{for $\leb^1$-a.e. $t$}.
\end{equation}
Applying the coarea formula to the function $u^*$ and taking into account \eqref{eq:derdistr} and \autoref{prop:identificationonedim}  we conclude that
\begin{align}\label{eq:expressiononedim}
\nonumber    \int_{0}^{r}\abs{\nabla u^*}_{1}^p\di \meas_{K,N} = \int_{0}^{r} \abs{(u^*)'}^p\di \meas_{K,N} &=\int_{0}^{ \sup u^*}\abs{(u^*)'((u^*)^{-1}(t))}^{p-1}\Per(\{u^*> t\})\di t\\
&=\int_{0}^{ \sup u^*}\frac{\left(\mathcal{I}_{K,N}(\mu(t))\right)^p}{\left(-\mu'(t)\right)^{p-1}}\di t.
\end{align}
Comparing \eqref{eq:boundfrombelow} with \eqref{eq:expressiononedim} we can conclude that
\begin{equation*}
\int_{\Omega}\abs{\nabla u}_1^p\di \meas\ge  \int_{0}^{r}\abs{\nabla u^*}_{1}^p\di \meas_{K,N}=\int_{0}^{r}\abs{\nabla u^*}^p\di \meas_{K,N},
\end{equation*}
giving \eqref{eq:polyaszego-1}.
To get \eqref{eq:polyaszego} it suffices to recall that $\abs{\nabla u}_1\le\abs{\nabla u}$ holds true  $\meas$-a.e..
\end{proof}

Armed with \autoref{thm:polyaszego} we can extend the celebrated Polya-Szego inequality to the non-smooth $\CD(K,N)$  framework.
\medskip

\textbf{Proof of \autoref{cor:symmapssobolevintosobolev}}.
\\By the very definition of $W^{1,p}_0(\Omega)$ we can find a sequence $(u_n)_{n\in\setN}$ with $u_n\in\LIPc(\Omega)$ for any $n\in\setN$ and $u_n$ converging to $u$ in $W^{1,p}(X,\dist,\meas)$, Moreover, thanks to \autoref{lemma:approximationwithnonvanishing}, we can assume that $\abs{\nabla u_n}_1\neq0$ for $\meas$-a.e. $x\in\set{u_n>0}$ for any $n\in\setN$, so that we can apply \autoref{thm:polyaszego} to each of the functions $u_n$ obtaining
\begin{equation}\label{eq:apppolya}
\int_{0}^{r}\abs{\nabla u^*_n}^p\di\meas_{K,N}\le \int_{\Omega}\abs{\nabla u_n}^p\di\meas.
\end{equation}
Observe now that the strong $L^p(X,\meas)$-convergence of $u_n$ to $u$ and the strong $L^p$-continuity of the monotone rearrangement (see \autoref{prop:equimeasurability}) grant that $u_n^{*}\to u^{*}$ in $L^{p}([0,r],\meas_{K,N})$. From the lower semicontinuity of the $p$-energy w.r.t. $L^{p}([0,r],\meas_{K,N})$-convergence it follows that:
\begin{equation*}
\int_{0}^{r} \abs{\nabla u^*}^p\di\meas_{K,N}\le\liminf_{n\to\infty}\int_{0}^{r} \abs{\nabla u^*_n}^p\di\meas_{K,N}.
\end{equation*}
Hence, taking into account \eqref{eq:apppolya} and the strong convergence in $W^{1,p}(X,\dist,\meas)$ of $u_n$ to $u$, we conclude that
\begin{equation*}
\int_{0}^{r}\abs{\nabla u^*}^p\di\meas_{K,N}\le \int_{\Omega}\abs{\nabla u}_w^p\di\meas,
\end{equation*}
which is the desired conclusion. 
\hfill$\Box$
\bigskip

In the following we will need an improved version of the Polya-Szego inequality. To this aim, for any non-negative $u\in W^{1,p}_0(\Omega)$ we introduce a function $f_u:[0,\sup u^*)\to[0,+\infty]$ by 
\begin{equation}\label{eq:deffu}
f_u(t):=\int\abs{\nabla u^*}^{p-1}\di\Per(\{u^*> t\}).
\end{equation}
Observe that this definition makes sense thanks to \autoref{cor:symmapssobolevintosobolev} and the coarea formula, which also yields
\begin{equation}\label{eq:fuCoarea}
\int_0^{\sup{u^*}}f_u(t)\di t=\int_{0}^{r}\abs{\nabla u^*}^p\di\meas_{K,N},
\end{equation}
for any $u\in W^{1,p}_0(\Omega)$.

We are now in position to state and prove our improved Polya-Szego inequalities.

\begin{proposition}[Improved Polya-Szego Inequalities]\label{prop:extensionofformulae}
Let $(X,\dist,\meas)$ be an essentially non branching $\CD(K,N)$ space for some $K>0$, $N\in (1,+\infty)$.  Let $\Omega\subset X$ be an open subset and consider  $[0,r]\subset I_{K,N}$ such that $\mm_{K,N}([0,r])=\mm(\Omega)$. 
\\Suppose that $u\in W^{1,p}_0(\Omega)$ is such that $u^{*}$ has non vanishing derivative $\leb^1$-a.e. on $(0,r)$. Then 
\begin{equation}\label{eq:improvedpolyawithper}
\int_{\Omega}\abs{\nabla u}^p\di\meas\ge \int_0^{\sup u^*}\left(\frac{\Per(\{u> t\})}{\mathcal{I}_{K,N}(\mu(t))}\right)^pf_u(t)\di t.
\end{equation}
As a consequence, under the same assumptions, it holds that
\begin{equation}\label{eq:improvedpolyawithprofile}
\int_{\Omega}\abs{\nabla u}^p\di\meas\ge\int_0^{\sup u^*}\left(\frac{\mathcal{I}_{(X,\dist,\meas)}(\mu(t))}{\mathcal{I}_{K,N}(\mu(t))}\right)^p f_u(t)\di t.
\end{equation}

\end{proposition}

\begin{proof}
In order to prove \eqref{eq:improvedpolyawithper} we just need to observe that our assumptions, even though being weaker than those of \autoref{thm:polyaszego}, put us in position to make its proof work.

Indeed, with the same notation therein introduced, we observe that the monotone rearrangement $u^*$ has the same distribution function of $u$. Moreover, \autoref{cor:symmapssobolevintosobolev} implies in particular that $u^{*}\in \AC_{\loc}((0,r))$.  Therefore, since we are assuming that $\abs{\nabla u^*}(t)\neq 0$ for $\leb^1$-a.e. $t$, it follows from \autoref{lemma:derivativedistribution} (taking into account also \autoref{prop:identificationonedim}) that $\mu$ is absolutely continuous and therefore differentiable $\leb^1$-a.e. with the explicit expression for the derivative given (for $\leb^1$-a.e. $t$) by
\begin{equation}\label{eq:derivativedistributionminimal}
-\mu'(t)=\frac{\Per(\{u^*> t\})}{\abs{\nabla u^*}((u^*)^{-1}(t))}=\frac{\mathcal{I}_{K,N}(\mu(t))}{\abs{\nabla u^*}((u^*)^{-1}(t))},
\end{equation}
where the second equality is a consequence of the very construction of the monotone rearrangement. 

Following verbatim the beginning of the proof of \autoref{thm:polyaszego} we obtain that \eqref{eq:levyapplied} is still valid in the present setting. Taking into account \eqref{eq:derivativedistributionminimal} we obtain that
\begin{align*}
-\varphi'(t)\ge\frac{\left(\Per(\{u> t\})\right)^p}{(-\mu'(t))^{p-1}}=&\frac{\left(\Per(\{u> t\})\right)^p}{\left(\mathcal{I}_{K,N}(\mu(t))\right)^{p-1}}\abs{\nabla u^*}((u^*)^{-1}(t))^{p-1}\\
=&\left[ \frac{\Per(\{u> t\})}{\mathcal{I}_{K,N}(\mu(t))}\right]^{p} \abs{\nabla u^*}((u^*)^{-1}(t))^{p-1}\mathcal{I}_{K,N}(\mu(t))\\
=&\left[ \frac{\Per(\{u> t\})}{\mathcal{I}_{K,N}(\mu(t))}\right]^{p}f_u(t)
\end{align*}
for $\leb^1$-a.e. $t\in(0,\sup u^*)$.
The desired inequality \eqref{eq:improvedpolyawithper} follows now recalling that
\begin{equation*}
\int_{\Omega}\abs{\nabla u}^p\di\meas=\int_0^{\sup u^*}(-\varphi'(t))\di t.
\end{equation*}

Conclusion \eqref{eq:improvedpolyawithprofile} is a consequence of \eqref{eq:improvedpolyawithper} after observing that $\{u> t\}$ is an admissible competitor in the definition of $\mathcal{I}_{(X,\dist,\meas)}(\mu(t))$ since, by definition, it holds that $\meas(\{u> t\})=\mu(t)$.
\end{proof}

\begin{remark}\label{rm:extensionofusefulformula}
In order to prove the forthcoming \autoref{thm:rigidityspectral} we will need to slightly enlarge the class of functions where \eqref{eq:improvedpolyawithper} and \eqref{eq:improvedpolyawithprofile} hold true.

In particular, we claim that \eqref{eq:improvedpolyawithper} holds true for any  $u\in W^{1,p}_0(\Omega)$ such that $u^*$ is $C^1$ and strictly decreasing. Indeed for any such $u$ it holds that the set of critical values of $u^*$ is $\leb^1$-negligible. Moreover, the distribution function $\mu$ of $u$ (which coincides with the distribution function of $u^*$ by equimeasurability, as we already observed), is differentiable at any regular point of $u^*$, with derivative given by  \eqref{eq:derivativedistributionminimal}. Hence the whole proof of \autoref{prop:extensionofformulae} can be carried over without modifications.

\end{remark}

\section{Spectral gap with Dirichlet boundary conditions}\label{sec:spectralgap}

\subsection{B\'erard-Meyer for essentially non-branching $\CD(K,N)$ spaces}
The goal of this section is to bound from below the $p$-spectral gap of an essentially non branching $\CD(K,N)$  space with the one of the corresponding one dimensional model space,
 for any $K>0, N\in (1,+\infty)$ and $p\in (1,+\infty)$. This extends to the non-smooth setting  the celebrated result of B\'erard-Meyer \cite{BerardMeyer} (see also \cite{Matei00}) proved for smooth Riemannian manifolds with $\Ric\geq K, K>0$.
\medskip

For every $K>0, N\in (1,+\infty)$, let $(I_{K,N}, \sfd_{eu}, \meas_{K,N})$ be the one dimensional model space defined in \eqref{eq:defonedim}. For every $v\in (0,1)$, let $r(v)\in I_{K,N}$ be such that $v=\meas_{K,N}([0,r(v)])$.
\\To let the notation be more compact, for any fixed $1<p<+\infty$, for any $v\in(0,1)$ and for any choice of $K>0$ and $1<N<+\infty$, we define
\begin{equation*}
\lambda^p_{K,N,v}:=\inf\left\lbrace\frac{\int_0^{r(v)} \abs{u'}^p\di \meas_{K,N}}{\int_0^{r(v)}u^p\di \meas_{K,N}}:\quad u\in\LIP([0,r(v)];[0,+\infty)),\quad u(r(v))=0 \text{ and $u\not\equiv 0 $} \right\rbrace, 
\end{equation*}
and we call $\lambda^p_{K,N,v}$ the \emph{comparison first eigenvalue for the $p$-Laplacian with Dirichlet boundary conditions for Ricci curvature bounded from below by $K$, dimension bounded from above by $N$ and volume $v$}.

Moreover, for any metric measure space $(X,\dist,\meas)$ with $\meas(X)=1$, for any open subset $\Omega\subset X$ and for any $1<p<+\infty$, we define

\begin{equation}\label{eq:firsteigenvalue}
\lambda_{X}^p(\Omega):=\inf\left\lbrace\frac{\int_{\Omega} \abs{\nabla u}^p\di \meas}{\int_{\Omega}u^p\di \meas}:\quad u\in\LIPc(\Omega;[0,+\infty))\text{ and $u\not\equiv 0$}\right\rbrace,
\end{equation}
and we call $\lambda^p_{X}(\Omega)$ the  \emph{first eigenvalue of the $p$-Laplacian on $\Omega$ with Dirichlet boundary conditions}.
\\Observe that for any $2\leq N\in\setN$ and $K>0$,  $\lambda^p_{K,N,v}=\lambda^{p}_{\mathbb{S}^{N}_{K}}(B_{v})$, where  $ \mathbb{S}^{N}_{K}$ is the round $N$-dimensional sphere or radius $\sqrt{\frac{N-1}{K}}$ and  $B_{v}\subset \mathbb{S}^{N}_{K}$ is a metric ball (i.e. a spherical cap) with  volume $v$.
\\We are now in position to prove \autoref{thm:pSpectr}
\bigskip

\textbf{Proof of \autoref{thm:pSpectr}}
\\For any $u\in\LIPc(\Omega;[0,+\infty))$ not identically zero we introduce the notation
\begin{equation*}
\mathcal{R}_p(u):=\frac{\int_{\Omega}\abs{\nabla u}^p\di \meas}{\int_{\Omega}u^p\di \meas}
\end{equation*}
for the Rayleigh quotient of $u$.

It follows from the combination of \autoref{prop:equimeasurability} and  \autoref{thm:polyaszego} that for any $u\in\LIPc(\Omega;[0,+\infty))$ such that $\abs{\nabla u}_1\neq 0$ $\meas$-a.e. on $\set{u>0}$ it holds
\begin{equation*}
\mathcal{R}_p(u)\ge\mathcal{R}_p(u^*),
\end{equation*}
where $u^*:[0, r(v)]\to[0,+\infty)$ is the monotone rearrangement of $u$ on the model space. Observe now that $u\in \LIP_{c}(\Omega)$ implies,  by construction of the monotone rearrangement, that $u^*$ vanishes at $r(v)$. We thus  get
\begin{equation*}
\mathcal{R}_p(u^*)\ge\lambda^p_{K,N,v}.
\end{equation*}
The desired conclusion follows from \autoref{cor:approximationnonvanishing1} and \autoref{rem:CDPoincare} which grant that for any $u\in\LIPc(\Omega;[0,+\infty))$ we can find a sequence $(u_n)_{n\in\setN}\subset \LIPc(\Omega;[0,+\infty))$ such that $\abs{\nabla u_n}_1\neq 0$ $\meas$-a.e. on $\set{u_n>0}$ for any $n\in\setN$ and 
\begin{equation*}
\mathcal{R}_p(u_n)\to\mathcal{R}_p(u),\quad \text{as $n\to\infty$}.
\end{equation*}
\hfill$\Box$

\subsection{Existence of minimizers}\label{sec:existenceofminimizers}
Here we collect some known result about the $p$-Laplace equation with homogeneous Dirichlet boundary conditions on metric measure spaces (verifying the curvature dimension condition) that will be useful in the next section about rigidity. We refer to \cite{LatvalaMarolaPere} and \cite{GigliMondino13} for a more detailed discussion about this topic and equivalent characterizations of first eigenfunctions.

Recall that we defined $W^{1,p}_0(\Omega)$ to be the closure w.r.t. the $W^{1,p}$-norm of $\LIPc(\Omega)$ (see \autoref{def:sobolevlocal}).
In the fairly general context of metric measure spaces it makes sense to talk about the first eigenfunction of the $p$-Laplace equation if the notion is understood in the following weak sense.

\begin{definition}[First eigenfunction]
Let $\Omega\subset X$ be an open domain. We say that $u\in W^{1,p}_0(\Omega)$ is a first eigenfunction of the $p$-Laplacian on $\Omega$ (with homogeneous Dirichlet boundary conditions) if $u\not\equiv 0$ and it minimizes the Rayleigh quotient
\begin{equation*}
\mathcal{R}_p(v)=\frac{\int_{\Omega}\abs{\nabla v}_w^p\di\meas}{\int_{\Omega}\abs{v}^p\di\meas},
\end{equation*}  
among all functions $v\in W^{1,p}_0(\Omega)$ such that $v\not\equiv 0$.
\end{definition}

\begin{remark}
Let us observe that if $u\in W^{1,p}_0(\Omega)$ is a first eigenfunction of the $p$-Laplacian then $\mathcal{R}_p(u)=\lambda^p_{X}(\Omega)$ (that is the first eigenvalue of the $p$-Laplace equation defined in \eqref{eq:firsteigenvalue}), since by the very definition of the space $W^{1,p}_0(\Omega)$ it makes no difference to minimize the Rayleigh quotient over $\LIPc(\Omega)$ or over $W^{1,p}_0(\Omega)$. As we will see below, the advantage of considering the minimization  over $W^{1,p}_0(\Omega)$ is to gain existence of minimizers.
\end{remark}

We conclude this short section with a general existence result for first eigenfunctions of the $p$-laplacian. The main ingredient for its proof, as in the smooth case, is the Sobolev inequality which implies in turn that also Rellich compactness theorem holds true in this setting. A good reference for this part is \cite[Chapter 5]{AmbrosioTilli04}.

\begin{theorem}[Existence of minimizers]\label{thm:existenceofminimizers}
Let $(X,\dist,\meas)$ be an essentially non branching $\CD(K,N)$ space, for some $K>0$ and $1<N<+\infty$.
Let $\Omega\subset X$ be an open subset, fix $1<p<+\infty$, and assume that $\lambda^p_{X}(\Omega)<+\infty$. 
\\Then  there exists a first eigenfunction of the $p$-Laplace equation (with homogeneous Dirichlet boundary conditions) on $\Omega$.
\end{theorem}

\begin{proof}
If $\lambda^p_{X}(\Omega)<+\infty$, we can find a sequence $(u_n)_{n\in\setN} \subset W^{1,p}_{0}(\Omega)$ such that $\norm{u_n}_{L^p}=1$ for any $n\in\setN$ and $\norm{\abs{\nabla u_n}_w}^p_{L^p}\to \lambda^p_{X}(\Omega)$ as $n\to\infty$.

Recall that the $\CD(K,N)$ condition for $K>0$ and $1<N<+\infty$ grants that $X$ is a compact and doubling metric measure space. Hence we can apply \cite[Theorem 5.4.3]{AmbrosioTilli04} (which is a general version of Rellich theorem for metric measure spaces) to the sequence $(u_n)_{n\in\setN}$ to find a limit function $u\in W^{1,p}_0(\Omega)$ such that $u_n\to u$ in $L^{p}(\Omega,\meas)$ as $n\to\infty$ and hence $\norm{u}_{L^p}=1$. It follows from the lower semicontinuity of the $p$-energy w.r.t. $L^{p}(\Omega,\meas)$-convergence that 
\begin{equation*}
\int_{\Omega}\abs{\nabla u}_w^p\di\meas\le\liminf_{n\to\infty}\int_{\Omega}\abs{\nabla u_n}_w^{p}\di\meas=\lambda^p_{X}(\Omega),
\end{equation*}
thus $u$ is a first eigenfunction of the $p$-laplacian with homogeneous Dirichlet boundary conditions on $\Omega$.
\end{proof}

\begin{remark}
Let us remark for sake of completeness that the definition of Sobolev space adopted in \cite{AmbrosioTilli04} is different with respect to the working one of this paper. However, as a consequence of \cite[Lemma 8.2]{AmbrosioColomboDiMarino}, if $(X,\dist,\meas)$ is an essentially non branching $\CD(K,N)$ m.m.s. and $f\in W^{1,p}(X,\dist,\meas)$ according to \autoref{def:sobolev}, then $f$ is a Sobolev function according to \cite[Definition 5.1.1]{AmbrosioTilli04} 
\end{remark}

\section{Rigidity}

\subsection{Rigidity in the Polya-Szego inequality}
This section is devoted to prove a rigidity statement associated to the  Polya-Szego inequality \autoref{thm:polyaszego}. The rough idea here is that if equality occurs in the Polya-Szego inequality then it occurs in the L\'evy-Gromov inequality too, and hence one can build on top of the rigidity statements in the L\'evy-Gromov isoperimetric inequality established in  \cite{CavallettiMondino17,CavallettiMondino18}. Let us also mention the paper \cite{FeroneVolpicelli}, where an elementary proof of the rigidity statement for the Polya-Szego inequality on $\setR^n$ is presented.

\begin{theorem}[Rigidity in the Polya-Szego inequality]\label{thm:RigPolSz} Let $(X,\dist,\meas)$ be an $\RCD(N-1,N)$ space for some $N\in[2,+\infty)$ with $\meas(X)=1$.  
\\Assume that there exists a  nonnegative function $u\in\LIP(X)$ achieving equality  in the Polya-Szego inequality \eqref{eq:polyaszego},  with $\abs{\nabla u}(x)\neq 0$ for $\meas$-a.e. $x\in\supp(u)$.

 Then $(X,\dist,\meas)$ is a spherical suspension, namely there exists an $\RCD(N-2,N-1)$ space $(Y,\dist_Y,\meas_Y)$ with $\meas_Y(Y)=1$ such that $(X,\dist,\meas)$ is isomorphic as a metric measure space to $[0,\pi]\times_{\sin}^{N-1}Y$.
  Moreover $u$ is radial, i.e.  $u=u^*(\dist(\cdot,x_{0}))$, where $x_{0}$ is a tip of a spherical suspension structure of $X$.
\end{theorem}

\begin{remark}\label{rem:uSignRigidity}
Before discussing the proof, let us stress that \autoref{thm:RigPolSz} is stated for a \emph{non-negative} function  $u$ just  for uniformity of notation with the previous sections. Nevertheless, such a non-negativity assumption can be suppressed, once the rearrangement $u^{*}$ in the  Polya-Szego inequality  \eqref{eq:polyaszego} is understood as the decreasing rearrangement of $|u|$ (see also \autoref{rem:RearSignChang}). The same holds for \autoref{thm:SpRigPoSz} below.
\end{remark}

\begin{proof}

\textbf{Step 1:} $(X,\dist,\meas)$ is a spherical suspension.
\\If equality occurs in \eqref{eq:polyaszego}, it follows from the proof of \autoref{thm:polyaszego} that  equality must occur in \eqref{eq:levyapplied} for $\leb^1$-a.e. $t\in(0,M)$, where $M:=\max u$. Hence  for $\leb^1$-a.e. $t\in(0,M)$ it holds:
\begin{equation}\label{eq:u>tOpt}
\Per(\{u> t\})=\mathcal{I}_{N-1,N}(\mu(t)).
\end{equation}
Since, by the very definition of the distribution function, we have $\meas(\{u> t\})=\mu(t)$, it follows that $\mathcal{I}_{(X,\dist,\meas)}(\mu(t))=\mathcal{I}_{N-1,N}(\mu(t))$ for $\leb^1$-a.e. $t\in(0,M)$. Thus we are in position to apply part $(i)$ of \autoref{thm:rigiditylevy} to conclude that $(X,\dist,\meas)$ is isomorphic to a spherical suspension $[0,\pi]\times_{\sin}^{N-1}Y$ for some $\RCD(N-2,N-1)$ space $(Y,\dist_Y,\meas_Y)$. 
\medskip

\textbf{Step 2:}  for every $t\in [0,M)$, the closure of the open superlevel set $\{u>t\}$ is a closed metric ball centred at a tip of a spherical suspension.
\\We first claim that \eqref{eq:u>tOpt} holds for every $t\in (0,M)$. To this aim, call $G\subset [0,M]$ the subset of those  $t\in (0,M)$ where \eqref{eq:u>tOpt} holds true. Since $G$ is dense,  for any fixed  $t\in[0,M]$ we can find a  sequence $(t_n)_{n\in\setN}\subset G$ such that  $t_n\to t$ as $n\to \infty$.  Our assumptions grant that $\{u> t_n\}$ converges in measure to $\{u> t \}$. From the lower semicontinuity of the perimeter and the continuity of the model isoperimetric profile it follows that:
\begin{align*}
\Per(\{u> t\})\leq \liminf_{n\to \infty}  \Per(\{u> t_{n}\} = \liminf_{n\to \infty} \mathcal{I}_{N-1,N}(\mu(t_{n})) = \mathcal{I}_{N-1,N}(\mu(t)),
\end{align*}
yielding the claim. In order to  conclude the proof of Step 2, observe that  $\{u>{t}\}$ is an open set, since $u$ is continuous. Denote the topological closure of $\{u>{t}\}$ by  $\overline{\{u> t\}}$.
Observe that, since $\supp(\mm)=X$ and $\abs{\nabla u}(x)\neq 0$ for $\meas$-a.e. $x\in \supp u$, $\meas(\overline{\{u> t\}}\setminus\{u>t\})=0$. Therefore the two sets have also the same perimeter. Hence part (iii) of \autoref{thm:rigiditylevy} implies that there exists an (a priori $t$-dependent) structure of spherical suspension $X\simeq [0,\pi]\times_{\sin}^{N-1}Y_t$ for a suitable $\RCD(N-2,N-1)$ space $(Y_t,\dist_t,\meas_t)$  such that  either
$$
\big[0,r(\mu(t))\big) \times Y_{t} \setminus \overline{\{u> t\}} \text{ or } \big(\pi-r(\mu(t)),\pi \big]\times Y_{t} \setminus\overline{\{u> t\}} 
$$
is an open set of $\mm$-measure zero, thus empty as $\supp \mm=X$. Without loss of generality we can assume  $[0,r(\mu(t)))\times Y_t\subset \overline{\{u>t\}}$.
\\Note that the topological closure of $\big[0,r(\mu(t))\big) \times Y_{t}$ is $\big[0,r(\mu(t))\big] \times Y_{t}$. Hence $\big[0,r(\mu(t))\big] \times Y_{t}\subset \overline{\{u>t\}}$. Moreover,  as we already observed above, we have that $\meas(\overline{\{u> t\}})=\meas(\set{u>t})=\meas \big([0,r(\mu(t))]\times Y_t \big)$. It follows that $\{u>t\}\subset [0,r(\mu(t))]\times Y_t$ (otherwise, since $ \{u>t\}\setminus [0,r(\mu(t))]\times Y_t$ is open, if non-empty  it would have strictly positive measure contradicting the last assertion) and hence $\overline{\{u>t\}}\subset [0,r(\mu(t))]\times Y_t$, as the right hand side is closed. Combining the inclusions obtained so far, we get
\begin{equation}\label{eq:ugeqtBall}
\big[0,r(\mu(t))\big] \times Y_{t}= \overline{\{u>t\}}.
\end{equation}
Let us stress that a priori the structure of spherical suspension may depend on $t\in (0,M)$; for instance in the $N$-sphere any point is a pole with respect to a corresponding structure of spherical suspension and any metric ball centred at any point is optimal for the isoperimetric problem.
\medskip

\textbf{Step 3:} Conclusion.
\\To prove the rigidity statement about the function $u$, we need to show that the above structure as spherical suspension is independent of $t\in (0,M)$. To this aim we first observe that, if equality holds in \eqref{eq:polyaszego}, then equality holds also in \eqref{eq:estimatepolya} for $\leb^1$-a.e. $t\in(0,M)$. Since \eqref{eq:estimatepolya} can be rewritten as
\begin{equation*}
\Per(\{u> t\})\le \left(\int\abs{\nabla u}^{p-1}\di \Per(\{u> t\})\right)^{\frac{1}{p}}\left(\int\frac{1}{\abs{\nabla u}}\di\Per(\{u> t \})\right)^{\frac{p-1}{p}},
\end{equation*}
we can conclude that, for $\leb^1$-a.e. $t\in[0,M]$,  $\abs{\nabla u}$ is constant $\Per(\{u> t \})$-a.e. by necessary conditions for equality in H\"older's inequality.
It follows that
\begin{equation}\label{eq:rigidityderivative}
\frac{1}{\abs{\nabla u}(x)}\mathcal{I}_{K,N}(\mu(t))=-\mu'(t), \quad \text{for $\leb^1$-a.e. $t\in [0,M]$ and  $\Per(\{u> t\})$-a.e. $x$.}
\end{equation}
Since $u^*$ is Lipschitz with $(u^*)'(t)<0$ for a.e. $t\in (0,r(v))$, it admits a strictly decreasing absolutely continuous  inverse function that we denote by  $v^*$. We claim that $f(x):=v^*\circ u(x)$ is the distance function from a fixed point $x_{0}$, playing the role of the pole of a structure  as spherical suspension \emph{independent of $t$}.  

To this aim we first observe that the combination of  \eqref{eq:rigidityderivative} with \eqref{eq:derivativedistributionminimal} gives that $\abs{\nabla f}(x)=1$ for $\meas$-a.e. $x\in\overline{\set{u>0}}$. From Step 2, we know that, for any $t\in (0,\max f)$, $ \overline{ \{f< t \}}$ is a closed metric ball of radius $r(t)$ centred at a point $x_{t}\in X$; moreover $X$ admits a structure of spherical suspension having $x_{t}$ as one of the two tips.  In particular, for any $t\in (0,\max f)$, it holds 
\begin{equation*}
\Per(\overline{ \{f< t \}})= \frac{1}{c_{N}} \sin(r(t))^{N-1}\text{ and } \meas(\overline{ \{f< t \}})= \frac{1}{c_{N}} \int_0^{r(t)}\sin(s)^{N-1}\di s.
\end{equation*}  
Combining what we obtained above with the assumption that $\abs{\nabla u}\neq 0$ $\meas$-a.e. on $\supp u$, we get that $u$ has a unique maximum point $x_0$, hence $\set{f\leq 0}=\set{x_0}$. In particular $r(0)=0$.
Taking into account the fact that $\abs{\nabla f}(x)=1$ for $\meas$-a.e. $x\in\supp u$, an application of the coarea formula yields that
\begin{equation}\label{eq:identityofvolumes}
 \frac{1}{c_N}\int_0^t\sin(r(s))^{N-1}\di s=\meas(\overline{ \{f< t \}})=\frac{1}{c_N}\int_0^{r(t)}\sin^{N-1}(s)\di s,
\end{equation} 
for every $t\in[0,\max f]$. Note that \eqref{eq:identityofvolumes} implies in particular that $t\mapsto r(t)$ is differentiable on $[0,\max f]$. Differentiating \eqref{eq:identityofvolumes} in $t$, we obtain that
\begin{equation*}
\sin(r(t))^{N-1}=r'(t)\sin(r(t))^{N-1},
\end{equation*}
for every $t\in[0,\max f]$. Therefore $r'(t)=1$ for every $t\in[0,\max f]$ and thus $r(t)\equiv t$. 
\\

\noindent 
We now claim that the centre $x_t$ of the ball $B_{t}(x_t)=\overline{\set{f< t}}$ is independent of $t$. 
\\If not we can find $t\in (0,\max f)$ such that $x_t\neq x_{0}$.
This implies that there exist $\epsilon>0$ and  $x\in \partial B_t(x_t)$ with $\dist(x_{0},x)\le t-\epsilon$. Since $f(0)=0$ and $f(x)=t$, we claim that this contradicts $\abs{\nabla f}=1$ $\meas$-a.e. on $\supp u$.
\\From the continuity of  $f$, we can find $\delta\in (0, \epsilon/8)$ such that for every $x'\in B_{\delta}(x)$ and every $y'\in B_{\delta}(x_{0})$ it holds $f(x')\ge f(x)-\epsilon/4=t-\epsilon/4$ and $f(y')\le f(x_0)+\epsilon/4=\epsilon/4$.
Now consider
\begin{equation*}
\mu_0:=\frac{1}{\meas(B_{\delta}(x))}\meas\res B_{\delta}(x), \quad \mu_1:=\frac{1}{\meas(B_{\delta}(x_{0}))}\meas\res B_{\delta}(x_{0})
\end{equation*}
and let $(\mu_t)_{t\in [0,1]}$ be a $W_2$-geodesic joining them. From \cite{GigliRajalaSturm16} the dynamic optimal transference plan $\nu$ representing  $(\mu_t)_{t\in [0,1]}$ is a test plan. We thus reach a contradiction:
\begin{equation*}
t-\frac{\epsilon}{2}\le \int f\di\mu_0-\int f\di\mu_1\le \int \abs{\nabla f}\abs{\dot{\gamma}}\di\nu\le t-\epsilon+2\delta\le t-\frac{3}{4}\epsilon.
\end{equation*}  
This proves that the center $x_t$ is independent of $t$ and  thus $\overline{\set{f< t}}=B_t(x_{0})$.
\\Since $f$ is continuous, it follows that $f(x)=t$ for every  $x\in \partial B_{t}(x_{0})$ and every $t\in [0, \max f]$, or in other words  $f(x)= \dist(x_{0},x)$ for every $x\in\supp(u)$.  The claim is given by composing  with $u^*$ both sides in this equality.
\end{proof}

\begin{remark}\label{rem:nablauneq0}
A natural question is  whether the condition $|\nabla u|\neq 0$ $\mm$-a.e. is sharp in \autoref{thm:RigPolSz}. Clearly, if $u$ is a constant function, also the decreasing rearrangement $u^{*}$ is constant; hence $u,u^{*}$  achieve equality in the Polya-Szego inequality but one cannot expect to infer anything on the space. However in the next \autoref{thm:SpRigPoSz}  we show that, as soon as $u$ is non constant, the equality in Polya-Szego forces the space to be a spherical suspension. The proof of such a statement is more delicate than step 1 of \autoref{thm:RigPolSz} and builds on top of the almost rigidity for L\'evy-Gromov inequality.
As already observed in \autoref{rem:nablauneq0Intro},  the condition $|\nabla u|\neq 0$ $\mm$-a.e. is necessary to infer that $u(\cdot)=u^{*}\circ \dist(x_{0}, \cdot)$, even knowing a priori that the space is a spherical suspension with pole $x_{0}$ and that $u$ achieves equality in Polya-Szego.
\end{remark}

\begin{theorem}[Space rigidity in the Polya-Szego inequality]\label{thm:SpRigPoSz}
 Let $(X,\dist,\meas)$ be an $\RCD(N-1,N)$ space for some $N\in[2,+\infty)$ with $\meas(X)=1$.  
\\Let $\Omega\subset X$ be an open set such that $\meas(\Omega)=v\in (0,1)$ and assume that there exists a nonnegative function $u\in W^{1,p}_0(\Omega)$, $u\not\equiv 0$, achieving equality  in the Polya-Szego inequality \eqref{PSIntro}.

 Then $(X,\dist,\meas)$ is a spherical suspension, namely there exists an $\RCD(N-2,N-1)$ space $(Y,\dist_Y,\meas_Y)$ with $\meas_Y(Y)=1$ such that $(X,\dist,\meas)$ is isomorphic as a metric measure space to $[0,\pi]\times_{\sin}^{N-1}Y$.
\end{theorem}

\begin{proof}

Let $(u_n)_{n\in\setN}$ be a sequence of Lipschitz functions with compact support in $\Omega$ such that $\abs{\nabla u_n}\neq 0$ $\meas$-a.e. on $\set{u_n>0}$ for any $n\in\setN$ approximating $u$ in $L^p(\Omega,\meas)$ and in $W^{1,p}$ energy given by \autoref{lemma:approximationwithnonvanishing}. Let $u_n^*$ and $u^*$ be the decreasing rearrangements of $u_n$ and $u$ respectively. The $L^p$-continuity of the decreasing rearrangement, together with the lower semicontinuity of the $p$-energy and the Polya-Szego inequality, yield
\begin{equation}\label{eq:liminflimsup}
\int_0^{r(v)}\abs{\nabla u^*}^p\di\meas_{N-1,N}\le\liminf_{n\to\infty}\int_0^{r(v)}\abs{\nabla u^*_n}^p\di\meas_{N-1,N}\le\limsup_{n\to\infty}\int_{\Omega}\abs{\nabla u_n}^p\di\meas=\int_{\Omega}\abs{\nabla u}^p\di\meas.
\end{equation}
It follows that $(u^*_n)_{n\in \setN}$ converges in $W^{1,p}$-energy to $u^*$, since by assumption $u$ achieves the equality in the Polya-Szego inequality.
\\ Up to extracting a subsequence, that we do not relabel, we can assume that $(u_n^*)_{n\in \setN}$ converges to $u^*$ both locally uniformly on $(0,r(v)]$ and in $W^{1,p}(([0,r(v)],\dist_{eu},\meas_{N-1,N}))$, and moreover that both the $\liminf$ and the $\limsup$ in \eqref{eq:liminflimsup} are full limits.
Denoting by $\mu_n$ and $\mu$ the distribution functions of $u_n$ and $u$ respectively, it follows that, for any $t\in(0,\sup u^*)$ such that $\meas_{N-1,N}(\set{u^*=t})=0$, it holds $\mu_n(t)\to\mu(t)$ as $n\to\infty$.
\\Moreover, if we let $f_n:=f_{u_n}$ be as in \eqref{eq:deffu}, then the improved Polya-Szego inequality \eqref{eq:improvedpolyawithprofile} grants that
\begin{equation*}
\int_{\Omega}\abs{\nabla u_{n}}^p\di\meas\ge\int_0^{\sup u^*_{n}}\left(\frac{\mathcal{I}_{(X,\dist,\meas)}(\mu_{n}(t))}{\mathcal{I}_{N-1,N}(\mu_{n}(t))}\right)^p f_{n}(t)\di t\geq \int_0^{\sup u^*_{n}} f_{n}(t)\di t=\int_{I_{N-1,N}}\abs{\nabla u^{*}_{n}}^p\di\meas_{N-1,N},
\end{equation*}
which, combined with the equality in the equality in \eqref{eq:liminflimsup}, gives

\begin{equation}\label{eq:contr}
\lim_{n\to\infty}\int_0^{\sup u_n^*}\left(\left(\frac{\mathcal{I}_{(X,\dist,\meas)}(\mu_n(t))}{\mathcal{I}_{N-1,N}(\mu_n(t))}\right)^p-1\right)f_n(t)\di t=0.
\end{equation}
Let us argue by contradiction and suppose that $(X,\dist,\meas)$ is not isomorphic to a spherical suspension. It follows from \autoref{thm:LevyGromov} that $\mathcal{I}_{(X,\dist,\meas)}(v)>\mathcal{I}_{N,N-1}(v)$ for any $v\in(0,1)$. Since by \autoref{prop:lscprofiles} we know that $\mathcal{I}_{(X,\dist, \meas)}$ is lower semicontinuous on $[0,1]$ and $\mathcal{I}_{N-1,N}$ is continuous on $[0,1]$ and positive on $(0,1)$, we have that, for any $0<\epsilon<1/2$, there exists $c_{\epsilon}>0$ such that
\begin{equation}\label{eq:belowprof}
\inf_{v\in[\epsilon,1-\epsilon]}\left\lbrace \left(\frac{\mathcal{I}_{(X,\dist,\meas)}(v)}{\mathcal{I}_{N-1,N}(v)}\right)^p-1\right\rbrace >c_{\epsilon}>0.
\end{equation} 
To conclude we observe that, thanks to the assumption that $u$ is non constant and to what we already observed, we can find $0<t_0<t_1<t_2<t_3<\sup u^*$, $0<\epsilon<1$ and $n_0\in\setN$ such that the following hold true: 
\begin{equation}\label{eq:nondegeneracy}
\int_{\set{t_1<u^*<t_2}}\abs{\nabla u^*}^p\di\meas_{N-1,N}>0,
\end{equation} 
\begin{equation}\label{eq:defcont}
\set{t_1<u^*<t_2}\subset\set{t_0<u_n^*<t_3}\quad\text{for any } n\ge n_0
\end{equation}
and
\begin{equation}\label{eq:defcont1}
\mu_n(t)\in[\epsilon,1-\epsilon]\quad \text{for any } t\in[t_0,t_3]\text{ and } n\ge n_0.
\end{equation}
Combining \eqref{eq:defcont} with the $L^{p}(\meas_{N-1,N})$ convergence of $\abs{\nabla u_n^*}$ to $\abs{\nabla u^*}$ and the coarea formula, we obtain that
\begin{equation}\label{eq:lastest}
\liminf_{n\to\infty}\int_0^{\sup u_n^*}f_n(t)\di t\ge\liminf_{n\to\infty}\int_{\set{t_0<u_n^*<t_3}}\abs{\nabla u_n^*}^{p}\di\meas_{N-1,N}\ge\int_{\set{t_1<u^*<t_2}}\abs{\nabla u^*}^{p}\di\meas_{N-1,N}.
\end{equation}
Eventually, putting \eqref{eq:belowprof} together with \eqref{eq:nondegeneracy} and \eqref{eq:lastest}, we obtain
\begin{equation*}
\liminf_{n\to\infty}
\int_0^{\sup u_n^*}\left(\left(\frac{\mathcal{I}_{(X,\dist,\meas)}(\mu_n(t))}{\mathcal{I}_{N-1,N}(\mu_n(t))}\right)^p-1\right)f_n(t)\di t\ge c_{\epsilon}\int_{\set{t_1<u^*<t_2}}\abs{\nabla u^*}^{p}\di\meas_{N-1,N}>0,
\end{equation*}
contradicting \eqref{eq:contr}.
\end{proof}

\autoref{thm:SpRigPoSzIntro}  follows from \autoref{thm:RigPolSz}, \autoref{rem:uSignRigidity} and \autoref{thm:SpRigPoSz}.

\subsection{Rigidity in the $p$-spectral gap}
The goal of this section is to investigate the rigidity in the $p$-spectral gap inequality (for Dirichlet boundary conditions), i.e. to prove  \autoref{thm:RigpSprectrIntro}.

\begin{theorem}[Domain-Rigidity for the $p$-spectral gap]\label{thm:rigidityspectral}
Let $(X,\dist,\meas)$ be an $\RCD(N-1,N)$ space. Let $\Omega\subset X$ be an open set with $\meas(\Omega)=v$ for some $v\in(0,1)$ and suppose that $\lambda^p_{X}(\Omega)=\lambda^p_{N-1,N,v}$.

Then $(X,\dist,\meas)$ is isomorphic to a spherical suspension and the topological closure of $\Omega$ coincides with the closed metric ball centred at one of the tips of the spherical suspension of $\meas$-measure $v$.
\end{theorem}

\begin{proof}
Suppose that $\lambda^p_{X}(\Omega)=\lambda^p_{N-1,N,v}$. Let $u\in W^{1,p}_0(\Omega)$ be a non-negative eigenfunction with $\|u\|_{L^{p}}=1$  associated to the first eigenvalue $\lambda^p_{X}(\Omega)$, whose existence is granted by \autoref{thm:existenceofminimizers}. Then  \autoref{cor:symmapssobolevintosobolev} gives:
\begin{equation*}
\lambda^{p}_{N-1,N,v}=\int_{\Omega}\abs{\nabla u}_w^p\di\meas\ge \int_0^{r}\abs{\nabla u^*}^p\di\meas_{N-1,N}\ge\lambda^p_{N-1,N,v},
\end{equation*} 
where, as before, $r$ is defined by $\mm_{N-1,N}([0,r])=\meas(\Omega)=v$. Hence equality holds true in all the inequalities so that $u^*$ is an eigenfunction of the $p$-Laplacian associated to the first eigenvalue on the one dimensional model space $\left([0,r],\dist_{eu},\meas_{N-1,N}\right)$. It follows from the corresponding ODE that $u^*\in C^{0}([0,r])\cap C^{1}((0,r))$ and it is strictly decreasing.

Hence, taking into account \autoref{rm:extensionofusefulformula}, \eqref{eq:improvedpolyawithper} holds true so that
\begin{equation*}
\lambda^{p}_{N-1N,v}=\int_{\Omega}\abs{\nabla u}_w^p\di\meas\ge \int_0^{\sup u^*}\left(\frac{\Per(\{u> t\})}{\mathcal{I}_{N-1,N}(\mu(t))}\right)^pf_u(t)\di t\ge\int_0^{\sup u^*}f_u(t)\di t=\lambda^p_{N-1,N,v}.
\end{equation*} 
Therefore it must hold 
\begin{equation}\label{eq:eqinLevy}
\Per(\{u> t\})=\mathcal{I}_{N-1,N}(\mu(t)), 
\end{equation}
for $\leb^1$-a.e. $t$ such that $f_u(t)\neq 0$.

In particular there exists at least one level $t_0$ such that the super-level set $\{u> t\}$ is optimal for the L\'evy-Gromov inequality. Thus we are in position to apply \autoref{thm:rigiditylevy} to conclude that $(X,\dist,\meas)$ is isomorphic, as a metric measure space, to a spherical suspension.

Moreover the $C^1$ regularity of $u^*$, together with Sard's lemma, grants that the set of those levels $t$ such that \eqref{eq:eqinLevy} holds true is dense in $(0,\sup u^*)$ (actually it is a full $\leb^1$-measure set). In particular we can find a sequence $(t_n)_{n\in\setN}$ monotonically converging to $0$ from above such that $\{u> t_n\}$ is optimal in the L\'evy-Gromov inequality. It follows from the continuity of the model profile $\mathcal{I}_{N-1,N}$ and the lower semicontinuity of the perimeter w.r.t. $L^1$ convergence that $\{u>0\}$ is optimal in the L\'evy-Gromov inequality itself.   Since $\{u>0\}$ is an open subset and  $\supp(\mm)=X$,  part (iii) of \autoref{thm:rigiditylevy} implies that there exists an $\RCD(N-2,N-1)$ space $(Y,\dist_{Y},\meas_Y)$ such that
  $X\simeq [0,\pi]\times_{\sin}^{N-1}Y$ and
$$
[0,r) \times Y \setminus \overline{\set{u> 0}} \text{ or } (\pi-r,\pi]\times Y\setminus \overline{\set{u> 0}} 
$$
is an open set of $\mm$-measure zero, hence it is empty (observe that $\overline{\{u>0\}}$ has the same perimeter and the same measure of $\{u>0\}$). Without loss of generality we can assume the first case holds and therefore $[0,r)\times Y\subset\overline{\set{u>0}}$.
\\Note that the topological closure of $[0,r) \times Y$ is $[0,r] \times Y$. Moreover, observing that \cite[Corollary 5.7]{LatvalaMarolaPere} grants that $u$ is strictly positive on $\Omega$, we obtain that the topological closure of 
 $\{u> 0\}$ coincides with $\bar{\Omega}$, the topological closure of $\Omega$. Applying once more part (iii) of \autoref{thm:rigiditylevy} and taking into account the assumption that $\meas$ has full support we can also say that $\set{u>0}\subset [0,r]\times Y$. The desired conclusion $\bar{\Omega}=[0,r]\times Y$ follows.
\end{proof}

Recall that, in the case of smooth Riemannian manifolds, the eigenfunction associated to the first eigenvalue on a smooth domain (with Dirichlet boundary conditions) is always simple (see for instance \cite{KawohlLindqvist06} for an elementary proof). In order to prove that in the case of rigidity in the spectral gap inequality also the eigenfunction must coincide with the radial one, we will exploit a generalization of such a principle to the case of our interest.

\begin{theorem}[Eingenfunction-Rigidity for the $p$-spectral gap]\label{prop:rigidityfortheeigenfucntion}
Let $(X,\dist,\meas)$ be an $\RCD(N-1,N)$ space  isomorphic to a spherical suspension and let $\Omega\subset X$ be an open  subset whose topological closure coincides with a closed metric ball centred at one of the tips of the spherical suspension satisfying  $\lambda^p_{X}(\Omega)=\lambda^p_{N-1,N,v}$,  $\meas(\Omega)=v\in (0,1)$.

Then, for any $1<p<+\infty$, the eigenfunction associated to the first eigenvalue of the $p$-Laplace equation with homogeneous Dirichlet boundary conditions on $\Omega$ is unique up to a scalar factor (and it coincides with the radial one).  
\end{theorem}

\begin{proof}
Since the closure of $\Omega$ coincides with the closed ball and $\Omega$ is open, we can infer that $\Omega$ is contained in the open ball. It follows that proving the statement we can suppose without loss of generality that $\Omega$ is an open ball centred at one of the tips of the spherical suspension. Indeed, any first Dirichlet eigenfunction on $\Omega$ is a first Dirichlet eigenfunction on the open ball. 

We wish to adapt the idea of \cite[Lemma 3.2]{KawohlLindqvist06} to this fairly more general setting: we claim that, if $v_1,v_2\in W^{1,p}_0(\Omega)$ are both non-negative eigenfunctions associated to the first eigenvalue of the $p$-Laplace equation with homogeneous Dirichlet boundary conditions on $\Omega$, defining $v:=(v_1^p+v_2^p)^{\frac{1}{p}}$ we obtain that $v\in W^{1,p}_0(\Omega)$ has Rayleigh quotient strictly smaller than $v_1$ (and $v_2$), unless $v_1$ and $v_2$ are proportional.
\medskip

\textbf{Step 1:} $v:=(v_1^p+v_2^p)^{\frac{1}{p}}\in W^{1,p}_0(\Omega)$ and is a first Dirichlet  $p$-eigenfunction.
\\Since $v_1,v_2\in W^{1,p}_0(\Omega)$ we can find sequences $(f_n)_{n\in\setN}$ and $(g_n)_{n\in\setN}$ of non-negative Lipschitz functions with compact support in $\Omega$ such that $f_n\to v_1$ and $g_n\to v_2$ in $L^p(\Omega,\meas)$ and $\int\abs{\nabla f_n}^p\di\meas\to\int\abs{\nabla v_1}_w^p\di\meas$ (and analogous statement for $v_2$). Introduce now $h_n:=(f_n^p+g_n^p)^{\frac{1}{p}}$. It is simple to check that $h_n\in\LIPc(\Omega)$ and $h_n\to v$ in $L^p(\Omega,\meas)$. Moreover the pointwise inequality
\begin{equation*}
\abs{\nabla{h_n}}^p(x)\le\abs{\nabla f_n}^p(x)+\abs{\nabla g_n}^p(x),
\end{equation*} 
holds true on $\Omega$ for any $n\in\setN$, hence $v\in W^{1,p}_0(\Omega)$. Furthermore it holds that
\begin{equation}\label{eq:optimalityofv}
\int_{\Omega}\abs{\nabla v}_w^p\di\meas\le\int_{\Omega}\abs{\nabla v_1}_w^p\di\meas+\int_{\Omega}\abs{\nabla v_2}_w^p\di\meas.
\end{equation}
Since $\norm{v}^p_{L^p}=\norm{v_1}^p_{L^p}+\norm{v_2}^p_{L^p}$ and $v_1,v_2$ are eigenfunctions associated to the first eigenvalue, equality actually holds true in \eqref{eq:optimalityofv} and $v$ is a first eigenfunction of the Dirichlet $p$-Laplacian.

The proof in the smooth case goes on by proving that it must hold $\abs{\nabla v}^p=\abs{\nabla v_1}^p+\abs{\nabla v_2}^p$ on $\Omega$ and then turning this information into the equality $\nabla \log v_1=\nabla\log v_2$ which gives the desired conclusion.
\\The non-smooth setting requires some care, in particular  we will call into play the notion of tangent module of a metric measure space (see \cite{Gigli14}).
\medskip

\textbf{Step 2:} It holds $ |\alpha \nabla \log v_1+ \beta \nabla \log v_2|_w= \alpha |\nabla \log v_1|_w+ \beta |\nabla \log v_2|_w,\,\meas\text{-a.e.}$,
where $\alpha:=\frac{v_1^p}{v_1^p+v_2^p}$  and $\beta:=\frac{v_2^p}{v_1^p+v_2^p}$; moreover $
|\nabla \log v_1|_w=|\nabla \log v_2|_w, \, \meas\text{-a.e.}$ .

Denote by $L^p_{\loc}(T\Omega)$ the module of locally $L^p(\meas)$-integrable vector fields on $\Omega$.
\\Observe  that \cite[Theorem 5.1]{LatvalaMarolaPere} grants continuity of $v_1,v_2$ on $\Omega$ (and therefore local boundedness), while \cite[Corollary 5.7]{LatvalaMarolaPere} ensures that they are also locally bounded away from $0$. Hence the following identity between elements of the $L^p_{\loc}$-normed $L^{\infty}$-module $L^{p}_{\loc}(T\Omega)$ makes sense and is justified by the chain rule:
\begin{equation*}
\nabla v=v\left(\frac{v_1^p}{v_1^p+v_2^p}\nabla\log v_1+\frac{v_2^p}{v_1^p+v_2^p}\nabla\log v_2\right).
\end{equation*}
An application of the defining properties of normed moduli and Jensen's inequality yield now that $\meas$-a.e. on $\Omega$ it holds
\begin{equation}\label{eq:chaineq}
\begin{split} 
\abs{\nabla v}_w^p\le& v^p\left(\frac{v_1^p}{v_1^p+v_2^p}\abs{\nabla\log v_1}_w+\frac{v_2^p}{v_1^p+v_2^p}\abs{\nabla\log v_2}_w\right)^p\\
\le&v^p\left(\frac{v_1^p}{v_1^p+v_2^p}\abs{\nabla\log v_1}_w^p+\frac{v_2^p}{v_1^p+v_2^p}\abs{\nabla\log v_2}_w^p\right)\\
=&\abs{\nabla v_1}_w^p+\abs{\nabla v_2}_w^p.
\end{split}
\end{equation}
Since $v,v_1,v_2$ are all first eigenfunctions it follows that equality holds true $\meas$-a.e. in all the above inequalities. 
Equality in \eqref{eq:chaineq} implies
\begin{equation}\label{eq:|aV1+bV2|=a|V1|+b|V2|}
|\alpha \nabla \log v_1+ \beta \nabla \log v_2|_w= \alpha |\nabla \log v_1|_w+ \beta |\nabla \log v_2|_w,  \quad  \meas\text{-a.e.},
\end{equation}
where $\alpha:=\frac{v_1^p}{v_1^p+v_2^p}$ and $\beta:=\frac{v_2^p}{v_1^p+v_2^p}$ satisfy $\alpha,\beta\in (0,1), \alpha+\beta=1$.
\\Moreover, equality in Jensen's inequality yields
\begin{equation}\label{eq:|Vlog1|=|Vlog2|}
|\nabla \log v_1|_w=|\nabla \log v_2|_w, \quad \meas\text{-a.e..}
\end{equation}

\textbf{Step 3:} in case $p=2$, then $\nabla \log v_1=\nabla \log v_2$ $\meas$-a.e. in the sense of $L^2_{\loc}$-modules.
\\Let us first consider the case $p=2$. In this case, equality in the triangle inequality \eqref{eq:|aV1+bV2|=a|V1|+b|V2|} forces equality in the Cauchy-Schwartz inequality 
\begin{equation*}
\langle \alpha \nabla \log v_1, \beta \nabla \log v_2  \rangle = |\alpha \nabla \log v_1|_w \, |\beta \nabla \log v_2|_w,  \quad \meas\text{-a.e.},
\end{equation*}
which in turn gives
\begin{equation}\label{eq:V1=gV2}
 \nabla \log v_1= \gamma \nabla \log v_2   \quad \meas\text{-a.e.},
\end{equation}
for some non-negative function $\gamma$. Combining \eqref{eq:V1=gV2} with \eqref{eq:|Vlog1|=|Vlog2|}, we get the desired conclusion

\begin{equation}\label{eq:V1=V2}
 \nabla \log v_1=  \nabla \log v_2   \quad \meas\text{-a.e.},
\end{equation}
in the sense of $L^2_{\loc}(T\Omega)$-modules.
\medskip

\textbf{Step 4:} Conclusion. 
\\Recall that, by assumption, $\bar{\Omega}=[0,R]\times T$ is a closed metric ball centred at a tip of the spherical suspension $X=[0,\pi]\times_{\sin}^{N-1} Y$,  for some $\RCD(N-2,N-1)$ space $(Y,\dist_Y,\meas_Y)$ and $\lambda_{X}^{p}(\Omega)=\lambda^{p}_{N-1,N,v}$.
\\In order to handle the case of a general $p\in (1,+\infty)$, observe that the function obtained as composition of the distance from a tip of a spherical suspension with the one dimensional first eigenfunction of the $p$-Laplacian on $([0, R], \dist_{eu}, \meas_{N-1,N})$ is a first eigenfunction on  $\Omega$, moreover we also know it has locally bounded gradient. Thus, choosing $v_1$ to be this particular eigenfunction,
the identity  \eqref{eq:|Vlog1|=|Vlog2|} gives that any other first $p$-eigenfunction on $\Omega$ has locally bounded gradient. In particular both $\nabla \log v_1$ and  $\nabla \log v_2$ belong to $L^2_{\loc}(T\Omega)$. Thus we reduced to the case $p=2$ and we conclude that \eqref{eq:V1=V2} holds.
\\Summarizing, we proved that if $v_1$ is the radial first eigenfunction and $v_2$ is any other first eigenfunction of the $p$-Laplacian, it holds $\nabla\log v_1=\nabla\log v_2$ as elements of $L^{p}_{\loc}(T\Omega)$. 
\\In order to conclude the proof we next show that $v_2\equiv cv_1$ for some constant $c\in \R$.
\\To this aim, for any $r_0<r_1\in (0,R)$  let 
$$
\Gamma^{r_{0}, r_{1}}:=\{\gamma^{r_{0}, r_{1}}_{y,\eps}\in \Geo(X): \, \gamma^{r_{0}, r_{1}}_{y,\eps}(t)=\big( t(r_{0}-\eps)+(1-t)(r_{1}-\eps), y \big), \, t\in [0,1], \, \eps\in [0, r_{0}/2 ], y\in Y \}.
$$
 Define $\nu^{r_1,r_2}\in \Prob(\Geo(X))$ by
$$
d\nu^{r_0,r_1}(\gamma):= \frac{1}{\meas([r_{0}/2,r_{0}]\times Y)}  \chi_{\Gamma^{r_0,r_1}}  \di\meas(\gamma(0)),
$$
where $ \chi_{\Gamma^{r_0,r_1}}$ is the characteristic function of $\Gamma^{r_0,r_1}$.
\\Since by assumption $\meas(t,y)=\sin^{N-1}(t) \leb^1(t)\otimes \meas_Y(y)$, it is easily seen that there exists $C_{r_1,r_2}\in (0,\infty)$ so that
$$
(\ee_t)_{\sharp}(\nu^{r_1,r_2})\leq C_{r_1,r_2} \meas, \quad \text{for all } t\in [0,1].
$$
Therefore $\nu^{r_1,r_2}$ is a test plan and we get
\begin{align*}
 \frac{1}{\meas([r_{0}/2,r_{0}]\times Y)}\int_{0}^{\frac{r_{0}}{2}} \int_Y  &|\log(v_1/v_2)(r_{1}-t,y)- \log(v_1/v_2)(r_0-t,y) | \di\meas_{Y}(y) \, \sin^{N-1}(t) \, \di t\\ 
&= \int_{\Geo(X)} |(\log(v_1/v_2))(\gamma_1)- (\log(v_1/v_2))(\gamma_0)| \di \nu^{r_0,r_1}(\gamma) \nonumber \\
&\leq  \int  |\nabla (\log v_{1}- \log v_2)|_w \di \nu^{r_0,r_1}(\gamma) =0.
\end{align*}
Hence, for all $r_0<r_1\in (0,R)$, we obtain that $\log(v_1/v_2)(r_0-t,y)=\log(v_1/v_2)(r_2-t,y)$ for $\meas_Y$-a.e. $y\in Y$ and $\leb^{1}$-a.e. $t\in (0, r_{0}/2)$. Since both $v_1$ and $v_2$ are continuous, we infer that there exists a continuous function $f:Y\to (0,+\infty)$ such that $v_1(t,y)=f(y) v_2(t,y)$ for all $(t,y)\in \Omega$.

Now, by chain rule, we get that $\log f\in W^{1,p}(Y,\dist_Y,\meas_Y)$ and  $ \nabla \log f (y)= \nabla \log v_1 (t,y)- \nabla \log v_2(t,y)=0 $ for $\meas_Y$-a.e. $y\in Y$ and $\leb^1$-a.e. $t\in (0,R)$. By the Sobolev-to-Lipschitz property holding on $\RCD(K,\infty)$ spaces \cite{AmbrosioGigliSavare14}, we get that $\log f\in \LIP(Y)$ with $\Lip(\log f)=0$; hence there exists $c\in \setR$ so that $f(y)=c$ for all $y\in Y$ and we conclude that $v_2\equiv c v_1$, as desired.
 \end{proof} 

The combination of \autoref{thm:rigidityspectral} and  \autoref{prop:rigidityfortheeigenfucntion} gives \autoref{thm:RigpSprectrIntro}.

\section{Almost rigidity in the Dirichlet $p$-spectral gap}
This section is dedicated to an almost-rigidity result which seems interesting even for smooth Riemannian manifolds. The idea is to argue by contradiction, exploiting on the one hand the compactness of the class of  $\RCD(N-1,N)$ spaces with respect to measured Gromov Hausdorff convergence and, on the other hand, the compactness/lower-semicontinuity of the functionals involved.
\\The following  result will play a key role in the compactness argument.
\begin{lemma}\label{lemma:lscofmeasures}
Let $(v_n)_{n\in\setN}$ be a sequence of functions in $W^{1,p}\left(([0,r],\dist_{eu},\meas_{N-1,N})\right)$ such that $v_n(r)=0$ for any $n\in\setN$. Assume that $(v_n)_{n\in\setN}$ converge in $L^p([0,r],\meas_{N-1,N})$ and in energy to $v\in W^{1,p}\left(([0,r],\dist_{eu},\meas_{N-1,N})\right)$. Define
\begin{equation*}
f_n(t):=\int\abs{\nabla v_n}^{p-1}\di\Per(\{v_n> t\}),\quad f(t):=\int\abs{\nabla v}^{p-1}\di\Per(\{v> t\})
\end{equation*} 
and let $\eta_n:=f_n\leb^1$ and $\eta:=f\leb^1$. Then $\eta_n\weakto\eta$ in duality with bounded and continuous functions.
\end{lemma}

\begin{proof}
We begin by observing that any function in $W^{1,p}\left(([0,r],\dist_{eu},\meas_{N-1,N})\right)$ is continuous in $(0,r]$. Indeed this result is well known in the case when, instead of $\meas_{N-1,N}$, the interval is equipped with the Lebesgue measure; in the case of our interest it suffices to observe that the density of $\meas_{N-1,N}$ w.r.t. $\leb^1$ is uniformly bounded from below on $[\epsilon,r]$ for any $\epsilon>0$. Moreover, by an analogous argument, functions in $W^{1,p}(([0,r],\dist_{eu},\meas_{N-1,N}))$ with uniformly bounded $p$-energies are uniformly H\"older continuous on $[\epsilon,r]$ for any $\epsilon>0$.

In view of what we remarked above, up to extracting a subsequence we can assume that $(v_n)_{n\in \setN}$ converges to $v$ uniformly on $[\epsilon,r]$ for any $\epsilon>0$ (recall that $v_n(r)=0$ for any $n\in\setN$). Moreover we can assume that the measures $\gamma_n:=\abs{\nabla v_n}^{p}\di\meas_{N-1,N}$ weakly converge to $\gamma:=\abs{\nabla v}^{p}\di\meas_{N-1,N}$.

We need to prove that for any bounded and continuous function $\phi:[0,+\infty)\to\setR$ it holds
\begin{equation}\label{eq:limitclaim}
\lim_{n\to\infty}\int\phi(t)f_n(t)\di t=\int\phi(t)f(t)\di t.
\end{equation}
To this aim we observe that, thanks to the coarea formula, it holds
\begin{align*}
\int\phi(t)f_n(t)\di t=&\int\phi(t)\left(\int\abs{\nabla v_n}^{p-1}\di\Per(\{v_n> t\})\right)\di t\\
=&\int\phi(v_n(x))\abs{\nabla v_n}^p(x)\di\meas_{N-1,N}(x)
\end{align*}
for any $n\in\setN$ (and an analogous identity holds true for $f$). Thus, in order to prove \eqref{eq:limitclaim}, it remains to prove that
\begin{equation}\label{eq:claim2}
\lim_{n\to\infty}\int_0^{r}\phi(v_n(x))\abs{\nabla v_n}^p(x)\di\meas_{N-1,N}(x)=\int_0^{r}\phi(v(x))\abs{\nabla v}^p(x)\di\meas_{N-1,N}(x).
\end{equation}
To this aim we observe that for any $\epsilon>0$ it holds that $\phi\circ v_n$ converge uniformly to $\phi\circ v$ on $[\epsilon,r]$, hence 
\begin{equation}\label{eq:goodcontrol}
\lim_{n\to\infty}\int_{\epsilon}^{r}\phi(v_n(x))\abs{\nabla v_n}^{p}(x)\di\meas_{N-1,N}(x)=\int_{\epsilon}^{r}\phi(v(x))\abs{\nabla v}^p(x)\di\meas_{N-1,N}(x).
\end{equation}
Moreover, calling $M:=\max\phi$, it holds that
\begin{equation}\label{eq:controltails}
\limsup_{n\to\infty}\abs{\int_0^{\epsilon}\phi\circ v_n\abs{\nabla v_n}^p\di\meas_{N-1,N}-\int_0^{\epsilon}\phi\circ v\abs{\nabla v}^p\di\meas_{N-1,N}}\le 2M\int_0^{\epsilon}\abs{\nabla v}^p\di\meas_{N-1,N}
\end{equation}
and the right hand-side in \eqref{eq:controltails} goes to $0$ as $\epsilon$ goes to $0$.
Therefore, in order to prove \eqref{eq:claim2}, it is sufficient to split the interval of integration into $[0,\epsilon]$ and $[\epsilon,r]$, pass to the $\limsup$ as $n\to\infty$ taking into account \eqref{eq:goodcontrol} and \eqref{eq:controltails} and then to let $\epsilon\downarrow 0$.
\end{proof}

A proof of the following useful result can be found for instance in \cite[Lemma 3.3]{AmbrosioGigliSavare15}. It will play a key role in the forthcoming proof of \autoref{thm:almostrigidity}.

\begin{lemma}[Joint lower semicontinuity]\label{lemma:jointlowersc}
Let $(\eta_n)_{n\in\setN}$ be a sequence of finite Borel measures weakly converging to a measure $\eta$ in duality w.r.t. bounded and continuous functions. Let moreover $f_n,f:\setR\to[0,+\infty)$ be such that
\begin{equation*}
f(t)\le\liminf_{n\to\infty}f_n(t_n)
\end{equation*}   
for any sequence $(t_n)_{n\in\setN}$ such that $t_n\to t$.
Then it holds that
\begin{equation*}
\int f(t)\di\eta(t)\le\liminf_{n\to\infty}\int f_n(t)\di\eta_n(t).
\end{equation*}
\end{lemma}

We can now prove the almost rigidity in the $p$-spectral gap \autoref{thm:almostrigidity}.
\medskip

\textbf{Proof of \autoref{thm:almostrigidity}}.
\\Let us argue by contradiction. If the conclusion is false there exist $\epsilon>0$, a sequence $(X_n)_{n\in\setN}$ of $\RCD(N-1,N)$  spaces (with $\meas_n(X_n)=1$) and open domains $\Omega_n\subset X_n$ such that $\meas_n(\Omega_n)=v$, $\lambda^p_{X}(\Omega_n)\le \lambda^p_{N-1,N,v}+\frac{1}{n}$ and 
\begin{equation}\label{eq:contrad}
\dist_{mGH}\left((X_n,\dist_n,\meas_n), (X,\dist,\meas)\right)\ge \epsilon
\end{equation}
for any spherical suspension $(X,\dist,\meas)$.

By the very definition of  $\lambda^p_{X}(\Omega)$ and thanks to the approximation result of \autoref{lemma:approximationwithnonvanishing}, for any $n\in\setN\setminus\{0\}$ we can find a nonnegative function $u_n\in\LIPc(\Omega_n)$ with $\abs{\nabla u_n}(x)\neq 0$ for $\meas_n$-a.e. $x\in\set{u_n>0}$ such that $\norm{u_n}_{L^p(\meas_n)}=1$ and 
\begin{equation*}
\int_{\Omega_n}\abs{\nabla u_n}^p\di\meas_n\le \lambda^p_{X}(\Omega_n)+\frac{1}{n}\le \lambda^p_{N-1,N,v}+\frac{2}{n}.
\end{equation*}

Call $\mu_n$ (respectively $f_n$) the distribution function of $u_n$ (respectively the function associated to $u_n$ as in \eqref{eq:deffu}). Recalling \eqref{eq:deffu}, \eqref{eq:fuCoarea} and applying \eqref{eq:improvedpolyawithprofile} to the function $u_n$ we obtain 
\begin{equation}\label{eq:consequence1}
\int_0^{r}\abs{\nabla u^*_n}^p\di\meas_{N-1,N}\le\int_0^{\sup u^*_n}\left(\frac{\mathcal{I}_{(X_n,\dist_{n},\meas_{n})}(\mu_n(t))}{\mathcal{I}_{N-1,N}(\mu_n(t))}\right)^p f_{n}(t)\di t\le\lambda^p_{N-1,N,v}+\frac{2}{n},
\end{equation}
where, as usual, $r$ is given by $\mm_{N-1,N}([0,r])=v$.
As a first consequence of \eqref{eq:consequence1} we obtain that, up to extracting a subsequence, $u^*_n$ weakly converges in $W^{1,p}\left(([0,r],\dist_{eu},\meas_{N-1,N})\right)$ to a function $u^*$. Moreover the convergence is uniform on $[\epsilon,r]$ for any $\epsilon>0$ so that in particular $u^*(r)=0$. By the lower semicontinuity of the $p$-energy, we know that 
\begin{equation*}
\int_0^{r}\abs{\nabla u^*}^p\di\meas_{N-1,N}\le \liminf_{n\to\infty}\int_0^{r}\abs{\nabla u^*_n}^p\di\meas_{N-1,N}\le \lambda^p_{N-1,N,v}.
\end{equation*} 
Hence $u^*$ is the first eigenfunction of the $p$-Laplacian on the model space $([0,r],\dist_{eu},\meas_{N-1,N})$ with unit $L^p$-norm satisfying $u^*(r)=0$. In particular $u^*_n$ converges to $u^*$ in $L^p$ and in $W^{1,p}$-energy. 
\\It follows that $u^*$ has negligible level sets so that, taking into account the local uniform convergence of the functions $u^*_n$ to $u^*$, we obtain the pointwise convergence of the distribution functions $\mu_n$ to the distribution function $\mu$ of $u^*$.
\\Moreover, using \autoref{lemma:lscofmeasures} we get that the sequence of measures $\eta_n:=f_n\leb^1$ weakly converges to $\eta:=f_{u^{*}}\leb^1$ in duality with bounded and continuous functions.
 
By compactness of the class of $\RCD(N-1,N)$ metric measure spaces w.r.t. measured Gromov Hausdorff convergence, there exists an $\RCD(N-1,N)$ space  $(X,\dist,\meas)$ such that (a subsequence of) $(X_n)_{n\in\setN}$ converges to it in the measured Gromov Hausdorff sense.
\\Introduce now functions $g_n$ and $g$ by
\begin{equation*}
g_n(t):=\left(\frac{\mathcal{I}_{(X_n,\dist_{n},\meas_{n})}(\mu_n(t))}{\mathcal{I}_{N-1,N}(\mu_n(t))}\right)^p,\quad g:=\left(\frac{\mathcal{I}_{(X,\dist,\meas)}(\mu(t))}{\mathcal{I}_{N-1,N}(\mu(t))}\right)^p,
\end{equation*}
for any $t\in [0,+\infty)$. \autoref{prop:lscprofiles}, together with the pointwise convergence of the distribution functions, yields that
\begin{equation}\label{eq:lsc}
g(t)\le\liminf_{n\to\infty}g_n(t_n)
\end{equation}
for any $t\in[0,\infty)$ and for any sequence $(t_n)_{n\in\setN}$ such that $t_n\to t$ as $n\to\infty$.
\\Applying \autoref{lemma:jointlowersc} with functions $g_n,g$ and measures $\eta_n$ and $\eta$, we conclude that
\begin{equation*}
\int_0^{\sup u^*}\left(\frac{\mathcal{I}_{(X,\dist,\meas)}(\mu(t))}{\mathcal{I}_{N-1,N}(\mu(t))}\right)^pf_{u^{*}}(t)\di t\le\liminf_{n\to\infty}\int_0^{\sup u_n^*}\left(\frac{\mathcal{I}_{(X_n,\dist_{n},\meas_{n})}(\mu_n(t))}{\mathcal{I}_{N-1,N}(\mu_n(t))}\right)^pf_n(t)\di t\le\lambda^p_{N-1,N,v},
\end{equation*}
where the last inequality follows from \eqref{eq:consequence1}.
\\Summarizing, we proved that
\begin{equation*}
\lambda^p_{N-1,N,v}=\int_0^{\sup u^*}f_{u^{*}}(t)\di t\le\int_0^{\sup u^*}\left(\frac{\mathcal{I}_{(X,\dist,\meas)}(\mu(t))}{\mathcal{I}_{N-1,N}(\mu(t))}\right)^pf_{u^{*}}(t)\di t\le\lambda^p_{N-1,N,v}.
\end{equation*} 
Hence it must hold 
\begin{equation*}
\mathcal{I}_{(X,\dist,\meas)}(\mu(t))=\mathcal{I}_{N-1,N}(\mu(t))
\end{equation*}
for at least one value of $t$ such that $\mu(t)\neq 0,1$. Therefore $(X,\dist,\meas)$ is isomorphic to a spherical suspension by \autoref{thm:rigiditylevy}. But this is in contradiction with \eqref{eq:contrad} since the sequence $(X_n, \dist_{n},\meas_{n})_{n\in\setN}$ is converging to $(X,\dist,\meas)$ in the mGH sense.
\hfill$\Box$


\begin{thebibliography}{GMS13}

{\footnotesize


\bibitem[Am02]{Am2}  \textsc{L.~Ambrosio}:
\textit{Fine properties of sets of finite perimeter in doubling metric measure spaces},
\newblock { Set-Valued Anal.}, \textbf{10}, (2002), 111--128.


\bibitem[Am18]{AmbrosioICM} 
	 \leavevmode\vrule height 2pt depth -1.6pt width 23pt:
  \textit{Calculus, heat flow and curvature-dimension bounds in metric measure spaces}, Proceedings of the ICM,
Rio de Janeiro, {\bf 1},  (2018), 301--340.
 
   

\bibitem[ACDM15]{AmbrosioColomboDiMarino}
      \textsc{L. Ambrosio, M. Colombo, S. Di Marino:}
      \textit{Sobolev spaces in metric measure spaces: reflexivity and lower semicontinuity of slope.}
      Advanced Studies in Pure Mathematics, {\bf 67}, (2015), 1--58.
   
   
\bibitem[ADM14]{AmbrosioDiMarino}
\textsc{L. Ambrosio, S. Di Marino:}
\textit{Equivalent definitions of {$\BV$} space and of total variation
              on metric measure spaces}, 
              J. Funct. Anal.,  {\bf 266}, (2014), 4150--4188.
		   
\bibitem[AGMR15]{AGMR12} \textsc{L.~Ambrosio, N.~Gigli, A.~Mondino, T.~Rajala},
\newblock Riemannian {R}icci curvature lower bounds in metric measure spaces with $\sigma$-finite measure,
\newblock{\em Trans. Amer. Math. Soc.,}   \textbf{367},  7,  (2015), 4661--4701.



\bibitem[AGS14a]{AmbrosioGigliSavare13}
	\textsc{L. Ambrosio, N. Gigli, G. Savar\'e}:
	\textit{Calculus and heat flow in metric measure spaces and applications to spaces with Ricci bounds from below},
	Invent. Math., \textbf{195}, (2014), 289--391.

\bibitem[AGS14b]{AmbrosioGigliSavare14}
	 \leavevmode\vrule height 2pt depth -1.6pt width 23pt:
	\textit{Metric measure spaces with Riemannian Ricci curvature bounded from below},
	Duke Math. J., \textbf{163}, (2014), 1405--1490.
	    
 \bibitem[AGS15]{AmbrosioGigliSavare15}
        \leavevmode\vrule height 2pt depth -1.6pt width 23pt:
       \textit{Bakry-\'Emery curvature-dimension condition and Riemannian Ricci curvature bounds},
       Annals of Probability, \textbf{43}, (2015), 339--404.
       


    \bibitem[AH16]{AmbrosioHonda}
	\textsc{L. Ambrosio, S. Honda:}
	\textit{New stability results for sequences of metric measure spaces with uniform Ricci bounds from below},
	Measure Theory in Non-Smooth Spaces, edited by Nicola Gigli, De Gruyter Press, Warsaw, 2017, 1--51.

       
\bibitem[AMS15]{AmbrosioMondinoSavare}
        \textsc{L. Ambrosio, A. Mondino, G. Savar\'e:}
        \textit{Nonlinear diffusion equations and curvature conditions in metric measure spaces,}
        Accepted paper at Memoirs Amer. Math. Soc.. Arxiv preprint 1509.07273.
       
       
\bibitem[APS15]{AmbrosioPinamontiSpeight}  
\textsc{L. Ambrosio, A. Pinamonti, G. Speight}:
\textit{Tensorization of {C}heeger energies, the space {$H^{1,1}$} and
              the area formula for graphs.} 
              Adv. Math.,  {\bf 281}, (2015), 1145-1177.
		
		
\bibitem[AT04]{AmbrosioTilli04}
	     \textsc{L. Ambrosio, P. Tilli:}
	     \textit{Topics on analysis on metric spaces,}
	     Oxford University Press, Oxford, {\bf 25}, (2004), viii--133.	
		
\bibitem[BS10]{BS10}  \textsc{K.~Bacher, K.-T. Sturm:}
\newblock Localization and tensorization properties of the curvature-dimension condition for metric measure spaces,
\newblock { \em J. Funct. Anal.,} \textbf{259}, (2010),  28--56.


\bibitem[BM82]{BerardMeyer}
     \textsc{P. B\'erard, D. Meyer:}
     \textit {In\'egalit\'es isop\'erim\'etriques et applications}, 
     Ann. Sci. \'Ecole Norm. Sup., {\bf 15}, 1982, 513-541.


\bibitem[Be05]{Bertrand05}
\textsc{J. Bertrand:} 
\textit{Stabilité de l'inégalité de Faber-Krahn en courbure de Ricci positive}, 
Annales de l'Institut Fourier, {\bf 55} (2005) no. 2, 353-372.


\bibitem[CaMi16]{CMi16} \textsc{F.~Cavalletti, E.~Milman}:
\newblock The Globalization Theorem for the Curvature Dimension Condition.
\newblock{\em preprint}  arXiv:1612.07623.



\bibitem[CM17a]{CavallettiMondino17}
          \textsc{F. Cavalletti, A. Mondino:}
          \textit{Sharp and rigid isoperimetric inequalities in metric-measure spaces with lower Ricci curvature bounds,}
          Invent. Math., {\bf 208}, (2017), 803-849.
          
          \bibitem[CM17b]{CMGT}
         \leavevmode\vrule height 2pt depth -1.6pt width 23pt:
          \textit{Sharp geometric and functional inequalities in metric measure spaces with lower Ricci curvature bounds,}
          Geom. Topol.,  \textbf{21}, (2017), 603--645.
         

\bibitem[CM18]{CavallettiMondino18}
            \leavevmode\vrule height 2pt depth -1.6pt width 23pt:
           \textit{Isoperimetric inequalities for finite perimeter sets under lower Ricci curvature bounds,}
           Atti Accad. Naz. Lincei Rend. Lincei Mat. Appl., \textbf{29},  (2018), no. 3, 413--430.
          

\bibitem[Ch99]{Cheeger}
         \textsc{J. Cheeger:} \textit{Differentiability of Lipschitz functions on metric measure spaces}. 
         Geom. Funct. Anal., \textbf{9}, (1999), 428--517.

\bibitem[EKS15]{EKS} \textsc{M.~Erbar, K.~Kuwada, K.T.~Sturm:}
\newblock \textit{On the Equivalence of the Entropic Curvature-Dimension Condition and Bochner's Inequality on Metric Measure Space},
\newblock {Invent. Math.}, \textbf{201}, (2015), no. 3, 993--1071.

\bibitem[Fa23]{Fa23} \textsc{G.~Faber}:
\textit{Beweiss dass unter allen homogenen Membranen von gleicher Fl\'ache und gleicher Spannung die kreisf\"ormgige den leifsten Grundton gibt},
{Sitz. bayer Acad. Wiss.}, (1923), 169--172.

\bibitem[FV03]{FeroneVolpicelli} \textsc{A. Ferone, R. Volpicelli}:
\textit{Minimal rearrangements of Sobolev functions: a new proof,}
Ann. Inst. H. Poincaré Anal. Non Linéaire, {\bf 20}, (2003), 333--339. 
        
      
\bibitem[G15a]{Gigli1} \textsc{N. Gigli:}
         \textit{On the differential structure of metric measure spaces and applications.}
      Mem. Am. Math. Soc., {\bf 236}, (2015), no. 1113, vi+91 pp.
		
\bibitem[G18]{Gigli14} \leavevmode\vrule height 2pt depth -1.6pt width 23pt:
		\textit{Nonsmooth differential geometry - An approach tailored for spaces with Ricci curvature bounded from below},
		Mem. Amer. Math. Soc., {\bf 251}, (2018), v+161 pp.


\bibitem[GH14]{GigliHan14}
       \textsc{N. Gigli, B. Han:}
       \textit{Independence on $p$ of weak upper gradients on $\RCD$ spaces},
      J. Funct. Anal., {\bf 271}, (2016), 1--11. 
       




 \bibitem[GM13]{GigliMondino13}
                \textsc{N. Gigli, A. Mondino:}
                \textit{A PDE approach to nonlinear potential theory in metric measure spaces}.
                 J. Math. Pures Appl., {\bf 100}, (2013), 503--534.



\bibitem[GMS15]{GigliMondinoSavare13}
	\textsc{N. Gigli, A. Mondino, G. Savar\'e:}
	\textit{Convergence of pointed non-compact metric measure spaces and
	stability of Ricci curvature bounds and heat flows},
	Proc. London Math. Soc. \textbf{111}, (2015), 1071--1129. 
	
	\bibitem[GRS16]{GigliRajalaSturm16}
	                \textsc{N. Gigli, T. Rajala, K.T. Sturm:}
	                \textit{Optimal maps and exponentiation on finite-dimensional spaces with Ricci curvature bounded from below,}
	                J. Geom. Anal., {\bf 26}, (2016), no. 4, 2914--2929. 
	
	
\bibitem[Gr07]{Gro}  \textsc{M.~Gromov}:
\textit{Metric structures for Riemannian and non Riemannian spaces}, Reprint of the 2001 English edition. Modern Birkh\"auser Classics. Birkh\"auser Boston, Inc., Boston, MA, (2007). xx+585 pp.
	
\bibitem[JZ16]{JZ16}	\textsc{Y. Jiang, H.-C. Zhang},
\textit{Sharp spectral gaps on metric measure spaces},
 Calc. Var. Partial Differential Equations, \textbf{55}, (2016), no. 1, Art. 14, 14 pp. 

\bibitem[KL06]{KawohlLindqvist06}
\textsc{B. Kawohl, P. Lindqvist:}
\textit{Positive eigenfunctions for the $p$-Laplace operator revisited,}
Analysis (Munich), {\bf 26}, (2006), no. 4, 545--550.


\bibitem[K06]{Kesavan}
        \textsc{S. Kesavan:}
        \textit{Symmetrization \& applications,} 
        Series in Analysis,
        volume 3, World Scientific Publishing Co. Pte. Ltd., Hackensack, NJ (2006),
        xii+148.

\bibitem[K15]{K15}   \textsc{C. Ketterer:}
       \textit{Obata's rigidity theorem for metric measure spaces.}
Anal. Geom. Metr. Spaces, \textbf{3}, (2015), 278--295.

       
\bibitem[Kl17]{klartag} \textsc{B. Klartag:}
\textit{Needle decomposition in Riemannian geometry.} Memoirs Amer. Math. Soc., {\bf 249}, no. 1180,
(2017).

\bibitem[Kr25]{Kr25} \textsc{E. Krahn}:
\textit{\"Uber eine von Rayleigh formulierte Minimaleigenschaftdes Kreises}, Math. Annalen, \textbf{94}, (1925), 97--100.

\bibitem[Kr26]{Kr26} \leavevmode\vrule height 2pt depth -1.6pt width 23pt:
\textit{\"Uber Minimaleigenschaften der Kugel in drei und mehr Dimensionen}, Acta Comm. Univ. Tartu (Dorpat), A9 (1926) pp. 1--44 (English transl.: \"U. Lumiste and J. Peetre (eds.), Edgar Krahn, 1894--1961, A Centenary Volume, IOS Press, 1994, Chap. 6, pp. 139-174).
 
\bibitem[LMP05]{LatvalaMarolaPere}
\textsc{V. Latvala, N. Marola, M. Pere:}
\textit{Harnack’s inequality for a nonlinear eigenvalue problem
on metric spaces},
J. Math. Anal. Appl., {\bf 321}, (2006), 793--810.


\bibitem[LV07]{LV07}\textsc{J. Lott, C. Villani}:
\textit{Weak curvature conditions and functional inequalities},
J. Funct.  Analysis, \textbf{245}, (2007), 311--333.

\bibitem[LV09]{LottVillani}
 \leavevmode\vrule height 2pt depth -1.6pt width 23pt:
\textit{Ricci curvature for metric-measure spaces via optimal transport}, 
Ann. of Math., \textbf{169} (2009), 903--991.
   

\bibitem[Ma00]{Matei00}
    \textsc{A. M. Matei:}
    \textit{First eigenvalue for the $p$-Laplace operator},
    Nonlinear Anal., \textbf{39}, (2000), 1051--1068. 


\bibitem[M15]{Milman15}
		\textsc{E. Milman:}
		\textit{Sharp isoperimetric inequalities and model spaces for the curvature-dimension-diameter condition,}
		J. Eur. Math. Soc., {\bf 17}, (2015), no. 5, 1041--1078. 

\bibitem[Mi03]{MirandaJr}
    \textsc{M. Miranda Jr.}:
     \textit {Functions of bounded variation on ``good'' metric spaces}, J. Math. Pures Appl.,
    {\bf 82}, (2003), 975--1004.

\bibitem[PS51]{PS51} \textsc{G. P\'olya, G. Szeg\"o}
\textit{Isoperimetric Inequalities in Mathematical Physics}. Annals of Mathematics Studies, \textbf{27}, Princeton University Press, Princeton, N. J., (1951). xvi+279 pp. 

\bibitem[RS14]{RS2014} {\sc T. Rajala,  K.T. Sturm}:
{\em Non-branching geodesics and optimal maps in strong $\CD(K,\infty)$-spaces},
Calc. Var. Partial Differential Equations, \textbf{50},  (2014), 831--846.

\bibitem[Ray]{Ray} {\sc J.~Rayleigh}:
{\em The Theory of Sound}, Macmillan, London, (1894). 

\bibitem[St06a]{Sturm06I} {\sc K.T.~Sturm}:
\newblock On the geometry of metric measure spaces. I,
\newblock{\em Acta Math.} \textbf{196} (2006), 65--131.

\bibitem[St06b]{Sturm06II}  \leavevmode\vrule height 2pt depth -1.6pt width 23pt:
\newblock On the geometry of metric measure spaces. II,
\newblock{\em Acta Math.} \textbf{196} (2006), 133--177.


\bibitem[V09]{Vil09}
\textsc{C. Villani:} 
\textit{Optimal transport. Old and New.} 
xxii–973, Grundlehren der Mathematischen Wissenschaften [Funda-
mental Principles of Mathematical Sciences], 338. Springer-Verlag Berlin, (2009).


\bibitem[V18]{VilB} \leavevmode\vrule height 2pt depth -1.6pt width 23pt: 
{\em In\'egalit\'es Isop\'erim\'etriques dans les espaces m\'etriques mesur\'es  [d'apr\`es F. Cavalletti \& A. Mondino]} S\'eminaire BOURBAKI
69\`me ann\'ee, 2016--2017, no. 1127. Available at http://www.bourbaki.ens.fr/TEXTES/1127.pdf.

\bibitem[VR08]{VonRenesse08}
    \textsc{M. K. Von Renesse:}
    \textit{On local Poincar\'e via transportation,}
    Math. Zeit., {\bf 259}, (2008), 21--31.
    
    

    \bibitem[W99]{Weaver}  \textsc{N. Weaver:}, 
    \textit{Lipschitz algebras}, World Scientific Publishing Co., Inc., River Edge, NJ, (1999).

}

\end{thebibliography}
\end{document}